\documentclass[11pt, a4paper]{article}
\usepackage[utf8]{inputenc}
\usepackage[usenames,dvipsnames]{xcolor}
\usepackage{amsmath}
\usepackage{amsthm}
\usepackage{amsfonts}
\usepackage{amssymb}
\usepackage{array}
\usepackage{graphicx} 
\usepackage{color}
\usepackage[english]{babel}
\usepackage{mathrsfs}
\usepackage{graphicx}
\usepackage{dsfont}
\definecolor{orangebis}{rgb}{0.99,0.25,0.00}
\definecolor{greenbis}{rgb}{0.10,0.85,0.10}
\definecolor{bluebis}{rgb}{0.10,0.30,0.99}
\usepackage[final]{hyperref}   % ICHANGED THIS ALL PARAGRAPH
\hypersetup{
    linktoc=page,
    linkcolor=red,          % color of internal links
    citecolor=blue,        % color of links to bibliography
    filecolor=blue,      % color of file links
    urlcolor=cyan,
   colorlinks=true      }     % color of external links}

\usepackage{multicol}
\usepackage{lipsum}
\usepackage{geometry}
\author{Hugo Vanneuville\thanks{Univ. Lyon 1, Institut Camille Jordan, 69100 Villeurbanne, France, supported by the ERC grant Liko No 676999}}
\title{Quantitative quenched Voronoi percolation and applications}
\date{}

\theoremstyle{plain}
\newtheorem{thm}{Theorem}[section]
\newtheorem{prop}[thm]{Proposition}
\newtheorem{lem}[thm]{Lemma}

\theoremstyle{definition}
\newtheorem{defi}[thm]{Definition}

\theoremstyle{remark}
\newtheorem{rem}[thm]{Remark}

\theoremstyle{remark}
\newtheorem{notation}[thm]{Notation}

\newcommand{\margin}[1]{\textcolor{magenta}{*}\marginpar{ \vskip -1cm \textcolor{magenta} {\it #1 }  }}
\renewcommand{\margin}[1]{}

\newcommand{\N}{\mathbb{N}}
\newcommand{\R}{\mathbb{R}}

\newcommand{\Z}{\mathbb{Z}}

\newcommand{\diam}{\text{\textup{diam}}}
\newcommand{\Pro}{\mathbb{P}}
\newcommand{\E}{\mathbb{E}}

\newcommand{\Var}{\mathbb{V}\text{\textup{ar}}}

\newcommand{\Prob}{\text{\textup{\textbf{P}}}}
\newcommand{\Ex}{\text{\textup{\textbf{E}}}}
\newcommand{\Piv}{\text{\textup{\textbf{Piv}}}}
\newcommand{\arm}{\text{\textup{\textbf{A}}}}
\newcommand{\cross}{\textup{\text{Cross}}}
\newcommand{\dense}{\textup{\text{Dense}}}
\newcommand{\Circ}{\textup{\text{Circ}}}
\newcommand{\qbc}{\textup{\text{QBC}}}

\newcommand{\gp}{\textup{\text{GP}}}

\newcommand{\setS}{\text{\textup{\textbf{S}}}}

\newcommand{\un}{\mathds{1}}

\newcommand{\grandO}[1]{O\mathopen{}\left(#1\right)}

\newcommand{\cond}{\, \Big| \,}
\renewcommand{\textbf}[1]{\begingroup\bfseries\mathversion{bold}#1\endgroup}
\def\diam{\mathrm{diam}}

\def\dist{\mathrm{dist}}

\def\E{\mathbb{E}} %{{\bf E}}

\def \eps {\epsilon}

\def\<#1{\langle #1\rangle}
\newcommand{\red}[1]{\textcolor{black}{#1}}
\newcommand{\blue}[1]{\textcolor{black}{#1}}

\def\bi{\begin{itemize}}  %USE\bi[WHATEVER]
\def\ei{\end{itemize}}
\def\bnum{\begin{enumerate}} % USE \bnum[i)] if want i), ii) .. OR \bnum[{\bf (a)}] etc ..!
\def\enum{\end{enumerate}}
\def\ni{\noindent}
\def\bf{\bfseries}

\numberwithin{equation}{section}
\begin{document}

\maketitle

\abstract{Ahlberg, Griffiths, Morris and Tassion have proved that, asymptotically almost surely, the quenched crossing probabilities for critical planar Voronoi percolation do not depend on the environment. We prove an analogous result for arm events. In particular, we prove that the variance of the quenched probability of an arm event is at most a constant times the square of the annealed probability. The fact that the arm events are degenerate and non-monotonic add two major difficulties. As an application, we prove that there exists $\epsilon > 0$ such that the following holds for the annealed percolation function $\theta^{an}$:
\[
\forall p > 1/2 ,\, \theta^{an}(p) \geq \epsilon (p-1/2)^{1-\epsilon} \, .
\]
One of our motivations is to provide tools for a spectral study of Voronoi percolation.
} 

%\abstract{In~\cite{ahlberg2015quenched}, Ahlberg, Griffiths, Morris and Tassion prove that, asymptotically almost surely, the quenched crossing probabilities for critical planar Voronoi percolation do not depend on the environment. We prove an analogous result for arm events; in particular, we prove that the variance of the quenched probability of an arm event is at most a constant times the square of the annealed probability. The fact that the arm events are degenerate and non-monotonic add two major difficulties. Our main tools are a martingale estimate in the spirit of~\cite{ahlberg2015quenched} and estimates on pivotal events from~\cite{scaling_voro}. As an application, we prove that there exists $\epsilon > 0$ such that the following holds for the annealed percolation function $\theta^{an}$:
%\[
%\forall p > 1/2 ,\, \theta^{an}(p) \geq \epsilon (p-1/2)^{1-\epsilon} \, .
%\]
%This is the analogue of the estimate proved by Kesten and Zhang~\cite{kesten1987strict} for Bernoulli percolation. We also explain why our main result is important for the study of noise sensitivity of Voronoi percolation. To this purpose, we introduce a random continuous point process inspired by the work of Garban, Pete and Schramm~\cite{garban2010fourier} which we call the annealed spectral sample of Voronoi percolation.} 

\tableofcontents

\section{Introduction}

\subsection{Main results}

\textbf{Planar Voronoi percolation} is a percolation model in random environment defined as follows (for more details, see for instance~\cite{bollobas2006critical,bollobas2006percolation} or the introduction of~\cite{scaling_voro}):

Let $p \in [0,1]$ and let $\eta$ be a homogeneous Poisson point process in $\R^2$ with intensity $1$. For each $x \in \eta$, let $C(x)=\{ u \in \R^2 \, : \, \forall y \in \eta, \, ||x-u||_2 \leq ||y-u||_2 \}$ be the Voronoi cell of $x$. \blue{We say that $x$ is the center of $C(x)$.} Note that a.s.\ all the Voronoi cells are bounded convex polygons. Given $\eta$, colour each cell in black with probability $p$ and in white with probability $1-p$, independently of the other cells. One thus obtains a random colouring of the plane. We write $\omega \in \{ -1,1 \}^\eta$ for the corresponding coloured configuration where $1$ means black and $-1$ means white, and we let $\Pro_p$ be the law of $\omega$. Let us be more precise about measurability issues. Let $\Omega'$ denote the set of locally finite subsets of $\R^2$ and let $\Omega=\cup_{\overline{\eta} \in \Omega'} \{-1,1\}^{\overline{\eta}}$. We equip $\Omega$ with the $\sigma$-algebra generated by the functions $\overline{\omega} \in \Omega \mapsto |\overline{\omega}^{-1}(1) \cap A|$ and $\overline{\omega} \in \Omega \mapsto |\overline{\omega}^{-1}(-1) \cap A|$ where $A$ spans the Borel subsets of the plane. The measure $\Pro_p$ is defined on this $\sigma$-algebra. 

We write $\{ 0 \leftrightarrow \infty \}$ for the event that there is a black path from $0$ to $\infty$ and we let $\theta^{an}(p)$ denote the \textbf{annealed percolation function} i.e.
\[
\theta^{an}(p)=\Pro_p \left[ 0 \leftrightarrow \infty \right] \, .
\]
The critical point of Voronoi percolation is
\[
p_c= \inf \{p \, : \, \theta^{an}(p) > 0 \} \, .
\]
Bollob\'{a}s and Riordan~\cite{bollobas2006critical} have proved that $p_c=1/2$. Duminil-Copin, Raoufi and Tassion~\cite{duminil2017exponential} have recently given an alternative proof of this result (and have even proved sharpness of Voronoi percolation in any dimension). The proof by Bollob\'{a}s and Riordan highly relies on a ``weak'' box-crossing property. A stronger box-crossing property has then been obtained by Tassion~\cite{tassion2014crossing}, see Theorem~\ref{t.tassion} below.

\blue{In the present paper, we are interested in \textit{quenched properties}. The quenched probability of an event is the probability of this event conditionally on the environment - i.e.\ conditionally on $\eta$. The annealed probability is the probability without any conditioning.} In~\cite{benjamini1999noise}, Benjamini, Kalai and Schramm have conjectured that, with high probability, the quenched crossing probabilities are very close to the annealed crossing probabilities. Ahlberg, Griffiths, Morris and Tassion have answered positively this conjecture in~\cite{ahlberg2015quenched}, see Theorem~\ref{t.AGMT} below. The results from~\cite{ahlberg2015quenched} provide very useful tools, that were for instance crucial in our work~\cite{scaling_voro} in which we have proved some scaling relations for Voronoi percolation. In the present paper, we pursue the work of~\cite{ahlberg2015quenched} by proving an analogue of their main theorem for \textbf{arm events} and by making their main result more quantitative. As consequences of the extension of~\cite{ahlberg2015quenched} to arm events, we will prove estimates on $4$-arm events and deduce a strict inequality for $\theta^{an}(p)$.
\medskip

Let us state the box-crossing results from~\cite{tassion2014crossing} and~\cite{ahlberg2015quenched}.
\begin{defi}
\bi 
\item[i)] For any $\lambda_1,\lambda_2 > 0$, $\cross(\lambda_1,\lambda_2)$ is the event that there is a black crossing of the rectangle $[-\lambda_1,\lambda_1] \times [-\lambda_2,\lambda_2]$ from left to right \blue{i.e.\ a black continuous path connecting the left side and the right side.}
\item[ii)] Given $\eta$, $\Prob^\eta_p$ is the conditional distribution of $\omega$ given $\eta$ i.e.\ $\Prob^\eta_p = \left( p\delta_1+(1-p)\delta_{-1}\right)^{\otimes \eta}$. \red{(More rigorously, for any measurable set $A \subseteq \Omega$, $\Pro_p [ A | \eta] = \Prob^\eta_p [ A \cap \{-1,1\}^\eta ]$.)} More generally, if $E$ is a countable set, we write $\Prob_p^E=\left( p\delta_1+(1-p)\delta_{-1}\right)^{\otimes E}$.
\ei
\end{defi}

\blue{By duality, $\Pro_{1/2} \left[\cross(n,n) \right]=1/2$ (to prove this, one needs to use that a.s.\ all the vertices of the Voronoi tiling have degree $3$)}. The following result is the annealed box-crossing property proved by Tassion:

\begin{thm}[Theorem~3 of~\cite{tassion2014crossing}]\label{t.tassion}
For every $\lambda \in (0,+\infty)$, there exists $c=c(\lambda) \in (0,1)$ such that, for every $R \in (0,+\infty)$,
\[
c \leq \Pro_{1/2} \left[ \cross(\lambda R,R) \right] \leq 1-c \, .
\]
\end{thm}

In~\cite{ahlberg2015quenched}, the authors prove a quenched box-crossing property in the case where $\eta$ is obtained by sampling $n$ independent uniform points in a rectangle. As mentionned in~\cite{ahlberg2015quenched} (see also Appendix~B of~\cite{scaling_voro}), the proof in the case where $\eta$ is a Poisson point process in $\R^2$ is essentially the same and we have the following:
\begin{thm}[\cite{ahlberg2015quenched}]\label{t.AGMT}
Let $\lambda > 0$. There exist an absolute constant $\epsilon > 0$ and a constant $C = C(\lambda) < +\infty$ such that, for every $R \in [1,+\infty)$,
\[
\Var \left( \Prob^\eta_{1/2} \left[ \text{\textup{Cross}}(\lambda R,R) \right] \right) \leq C R^{-\epsilon} \, .
\]
\end{thm}

\paragraph{Main results.} In the present paper, we prove an analogue of Theorem~\ref{t.AGMT} for arm events. Let us first define these events. Let $j \in \N^*$ and $0 \leq r \leq R$. The $j$-arm event from distance $r$ to distance $R$ is the event that there exist $j$ paths of alternating colors in the annulus $[-R,R]^2 \setminus (-r,r)^2$ from $\partial [-r,r]^2$ to $\partial [-R,R]^2$ (if $j$ is odd, we ask that there are: (a) $j-1$ paths of alternating color, and (b) one additional black path such that there is no Voronoi cell intersected by both this additional path and one of the $j-1$ other paths). Let $\arm_j(r,R)$ denote this event. \blue{(If $r > R$, we let $\arm_j(r,R)$ be the sure event.)} The annealed probability of $\arm_j(r,R)$ is denoted by
\[
\alpha^{an}_{j,p}(r,R) = \Pro_p \left[ \arm_j(r,R) \right] \, .
\]
We will use the simplified notation $\alpha_{j,p}^{an}(R)=\alpha_{j,p}^{an}(1,R)$. Our main theorem is the following:
\begin{thm}\label{t.quenched_arm}
Let $j \in \N^*$. There exists a constant $C = C(j) < +\infty$ such that, for every $r,R \in [1,+\infty)$ that satisfy $r \leq R$, we have
\begin{equation}\label{e.main2}
\alpha_{j,1/2}^{an}(r,R)^2 \leq \E \left[ \Prob^\eta_{1/2} \left[ \arm_j(r,R) \right]^2 \right] \leq C \, \alpha_{j,1/2}^{an}(r,R)^2 \, .
\end{equation}
Let also $a \in (0,1)$. There exists a constant $C'=C'(j,a)<+\infty$ such that, if we assume furthermore that $r \leq aR$, then
\begin{multline}\label{e.main}
\E \left[ \Prob^\eta_{1/2} \left[ \arm_j(r,R) \right]^2 \right] - \alpha^{an}_{j,1/2}(r,R)^2 = \Var \left( \Prob^\eta_{1/2} \left[ \arm_j(r,R) \right] \right)\\
\leq C' \, \alpha_{j,1/2}^{an}(r,R)^2 \, r^2 \, \alpha_{4,1/2}^{an}(r)^2 \, .
\end{multline}
\end{thm}

\begin{rem}
The estimate~\eqref{e.main2} of Theorem~\ref{t.quenched_arm} is a direct consequence of~\eqref{e.main} and of an estimate on the $4$-arm events proved in~\cite{scaling_voro} (see Proposition~\ref{p.alpha_4} of the present paper). However, our strategy will be to first prove~\eqref{e.main2} and then deduce~\eqref{e.main}.
\end{rem}

The new difficulties compared to the work~\cite{ahlberg2015quenched} are the fact that the arm events are \textbf{degenerate} and (except for $j=1$) \textbf{non-monotonic}. The fact that the crossing events are monotonic was crucial in~\cite{ahlberg2015quenched}, especially in Section~$2$ where the authors prove an Efron-Stein estimate by revealing the position of the points of $\eta$ one after the other, and in their final section where they use the Schramm-Steif randomized algorithm theorem \cite{schramm2010quantitative} \blue{in order to estimate the sum of squares of influences.}\footnote{\blue{Consider the hypercube $\{-1,1\}^n$ equipped with the uniform probability measure and let $A \subseteq \{-1,1\}^n$. The influence of a coordinate $i \in \{1,\cdots,n\}$ is the probability that, if we change the value of the $i^{th}$ coordinate, then this modifies the Boolean function $\un_A$. The Schramm-Steif theorem is an estimate about the Fourier decomposition of Boolean functions. This theorem also holds for non-monotonic functions. However, the connection between the influences and the Fourier spectrum that is used in \cite{ahlberg2015quenched} is only true for monotonic functions.}} To deal with these new difficulties, we will have to use very precise estimates on the \textbf{pivotal events}. By doing so, we will also obtain the following more quantitative version of Theorem~\ref{t.AGMT}:
\begin{thm}\label{t.quenched_cross}
Let $\lambda > 0$. There exists a constant $C = C(\lambda) < +\infty$ such that, for every $R \in (0,+\infty)$,
\begin{equation}\label{e.maincross}
\Var \left( \Prob^\eta_{1/2} \left[ \text{\textup{Cross}}(\lambda R,R) \right] \right) \leq C R^2 \, \alpha_{4,1/2}^{an}(R)^2 \, .
\end{equation}
\end{thm}

We refer to Subsection~\ref{ss.idea} for some intuitions and ideas of proofs.
%Before that, let us say very briefly that in~\cite{ahlberg2015quenched}, the authors bound a variance by a sum of squares of probabilities of ``pivotal events''. In Bernoulli percolation, squares of robabilities of ``pivotal events'' behave like the analogues of the right-hand sides of~\eqref{e.main} and~\eqref{e.maincross}.

\begin{rem}
An interesting question is whether or not Theorems~\ref{t.quenched_arm} and~\ref{t.quenched_cross} are optimal: Is Theorem~\ref{t.quenched_arm} (respectively Theorem~\ref{t.quenched_cross}) still true with $r^{2-\epsilon}$ (respectively $R^{2-\epsilon}$) instead of $r^2$ (respectively $R^2$)? It is likely that the general martingale estimate Proposition~\ref{p.martingale} is not optimal at all in the case of crossing (and arm) events.
\end{rem}

\subsection{An application: Reimer's inequality and the annealed percolation function}

In this subsection, we explain how one can use~\eqref{e.main2} in order to obtain \red{some} estimates on \red{the} annealed probabilities of arm events. We first need to define the disjoint occurrence of two events. If $\omega$ is a coloured configuration, we write $\eta(\omega)$ for the underlying non-coloured \red{set of points (i.e.\ $\eta(\omega)$ is such that $\omega \in \{-1,1\}^{\eta(\omega)}$)}. \red{Recall that $\Omega$ is the set of all coloured configurations. If $A,B \subseteq \Omega$ are measurable with respect to the coloured configuration $\omega$ restricted to a bounded domain, we write
\begin{equation}\label{e.disjoint_occ}
A \square B = \left\lbrace \omega \in \Omega  \; : \; \exists I_1,I_2 \text{ finite disjoint subsets of } \eta(\omega), \; \omega^{I_1} \subseteq A \text{ and } \omega^{I_2} \subseteq B \right\rbrace  \, ,
\end{equation}
where, for all $I \subseteq \eta(\omega)$,}
\[
\omega^I = \{ \omega' \in \lbrace -1,1 \rbrace^{\eta(\omega)} \, : \, \forall i \in I, \, \omega'_i=\omega_i \}.
\]
By Reimer's inequality~\cite{reimer2000proof} (which generalizes the BK inequality to non-necessarily monotonic events, see for instance~\cite{grimmett1999percolation,bollobas2006percolation}), we have the following quenched inequality:
\[
\Prob^\eta_p \left[ A \square B \right] \leq \Prob^\eta_p \left[ A \right] \Prob^\eta_p \left[ B \right] \, .
\]
However, the analogous annealed property ``$\Pro_p \left[ A \square B \right] \leq \Pro_p \left[ A \right] \Pro_p \left[ B \right]$'' is not true. Indeed, if $A$ depends only on $\eta$ and satisfies $\Pro \left[ A \right] \in (0,1)$, then $\Pro \left[ A \right] = \Pro \left[ A \square A \right] > \Pro \left[ A \right]^2$ \red{(indeed, one can just choose $I_1=I_2=\emptyset$ since for any set $I \subseteq \eta(\omega)$ and any $\omega' \in \omega^I$, we have $\eta(\omega')=\eta(\omega$))}. Let us note that, if $A$ and $B$ are annealed increasing (which means that they are stable under addition of black points and deletion of white points) and if $p=1/2$, then the annealed property $\Pro_{1/2} \left[ A \square B \right] \leq \Pro_{1/2} \left[ A \right] \Pro_{1/2} \left[ B \right]$ holds. This is the annealed BK inequality, see Lemma~3.4 of~\cite{ahlberg2015quenched} or~\cite{joosten2012random}.

\red{By the above observation, we can expect that the annealed Reimer inequality does not hold for the events that highly depend on the environment $\eta$, even up to a constant. On the contrary, the estimate \eqref{e.main2} - which implies that the arm events depend little on $\eta$ - can be used to prove that the arm events satisfy an annealed Reimer inequality at $p=1/2$ up to a constant. For instance, for every $j \in \N^*$ we have}
\begin{align*}
& \alpha_{2j+1,1/2}^{an}(r,R)\\
& = \E \left[ \Prob^\eta_{1/2} \left[ \arm_1(r,R) \square \arm_{2j}(r,R) \right] \right]\\
& \leq \E \left[ \Prob^\eta_{1/2} \left[ \arm_1(r,R) \right] \Prob^\eta_{1/2} \left[ \arm_{2j}(r,R) \right] \right] \text{ by Reimer's inequality}\\
& \leq \sqrt{\E \left[ \Prob^\eta_{1/2} \left[ \arm_1(r,R) \right]^2 \right] \E \left[ \Prob^\eta_{1/2} \left[ \arm_{2j}(r,R) \right]^2 \right]} \text{ by the Cauchy-Schwarz inequality}\\
& \leq \grandO{1} \alpha_{2j,1/2}^{an}(r,R) \alpha_{1,1/2}^{an}(r,R) \text{ by~\eqref{e.main2} \, .}
\end{align*}
It seems complicated to prove this estimate without relying on~\eqref{e.main2}. Actually, still by relying on~\eqref{e.main2}, we will prove in Section~\ref{s.strict} that $\alpha_{2j+1,1/2}^{an}(r,R) \leq \grandO{1} \left( \frac{r}{R} \right)^{\epsilon} \alpha_{2j,1/2}^{an}(r,R) \alpha_{1,1/2}^{an}(r,R)$. This identity will be a key result in order to prove the following strict inequality for the annealed percolation function, which is analogous to the result obtained by Kesten and Zhang in~\cite{kesten1987strict}:
\begin{thm}\label{t.strict}
There exists a constant $\epsilon > 0$ such that, for every $p > 1/2$, we have
\[
\theta^{an}(p) \geq \epsilon \left( p-1/2 \right)^{1-\epsilon} \, .
\]
\end{thm}
Let us note that the authors of~\cite{duminil2017exponential} have obtained that $\theta^{an}(p) \geq \epsilon \, (p-1/2)$ \textbf{in any dimension}. Theorem~\ref{t.strict} is proved in Section~\ref{s.strict}. In order to prove this result, we also rely on the following two annealed scaling relations (analogous to the scaling relations proved by Kesten for Bernoulli percolation on $\Z^2$, see~\cite{kesten1987scaling}) that we have proved in~\cite{scaling_voro}:
\begin{thm}[Theorem~1.11 of~\cite{scaling_voro}]\label{t.scaling_rel}
For every $p \in (1/2,3/4]$, let $L^{an}(p)$ denote the annealed correlation length, i.e.
\[
L^{an}(p) = \inf \{ R \geq 1 \, : \, \Pro_p \left[ \cross(2R,R) \right] \geq 1-\eps_0 \} < +\infty \, ,
\]
for some fixed sufficiently small $\eps_0$ (see Subsection~1.3 of~\cite{scaling_voro} for more details). \red{There exist $c=c(\eps_0)\in (0,+\infty)$ and $C=C(\eps_0) \in (0,+\infty)$ such that, for every $p \in (1/2,3/4]$,}
\[
c \, \alpha^{an}_{1,1/2}(L^{an}(p)) \le \theta^{an}(p) \le C \, \alpha^{an}_{1,1/2}(L^{an}(p))
\]
and
\[
c \, \frac{1}{p-1/2} \le L^{an}(p)^2 \alpha_{4,1/2}^{an}(L(p)) \le C \, \frac{1}{p-1/2} \, .
\]
\end{thm}

\subsection{A motivation: noise sensitivity and exceptional times for Voronoi percolation}

One of the main motivations of the present paper (and in particular of Theorem \ref{t.quenched_arm}) is to provide tools to study the \textbf{annealed spectral sample of Voronoi percolation}, which is a continuous spectral object that we introduce and study in \cite{Vannealed}. \blue{\red{Define a dynamical Voronoi percolation process} by either resampling each colour of $\omega$ at rate $1$ or letting the points of $\eta$ evolve according to (independent) long range Lévy processes. In the first case, the colours evolve in time while in the second case the ``environment'' evolves in time. By studying the annealed spectral sample, we prove in \cite{Vannealed} that in both dynamical processes a.s.\ there exist \textbf{exceptional times}, i.e.\ times with an unbounded black component.}

The present paper can actually be seen as the second of a series of three papers whose first is \cite{scaling_voro} (in which we prove the estimates stated in Subsection \ref{ss.a_and_piv}), whose third is \cite{Vannealed}, and whose final goal is to study the annealed spectral sample of Voronoi percolation.

\blue{Why is the present paper important for the study of annealed spectral objects? In \cite{garban2010fourier}, Garban, Pete and Schramm study the so-called (discrete) spectral sample. To this purpose, they rely deeply on general spectral inequalities that involve \textbf{the square of the probability of pivotal events}. The analogous spectral inequalities for the annealed spectral sample (see \cite{Vannealed}) involve the expectation of the square of the quenched probability of pivotal events. In the case of percolation, the latter quantity is highly connected to the quantity $ \E \left[ \Prob^\eta_{1/2} \left[ \arm_4(r,R) \right]^2 \right]$ that is studied in the present work.}

The study of noise sensitivity of Voronoi percolation \blue{(which is a notion intimately related to the question of existence of exceptional times)} has been initiated in~\cite{ahlberg2015quenched} and~\cite{ahlberg2017noise}.

\subsection{Ideas of proof}\label{ss.idea}

\begin{notation}
In the paper, we will only work at the parameter $p=1/2$ (the scaling relations of Theorem~\ref{t.scaling_rel} enable us to estimate $\theta(p)$ with $p > 1/2$ by working at $p=1/2$). Hence, we will use the following simplified notations:
\bi 
\item $\Pro := \Pro_{1/2}$, $\Prob^\eta := \Prob_{1/2}^\eta$ and $\Prob^E := \Prob_{1/2}^E \, ,$
\item $\alpha_j^{an}(r,R) := \alpha^{an}_{j,1/2}(r,R) \, .$
\ei
Also, we will use the following notation:
\begin{equation}\label{e.notation_alphatilde}
\widetilde{\alpha}_j(r,R) = \sqrt{ \E \left[ \Prob^\eta \left[ \arm_j(r,R) \right]^2 \right]} \, .
\end{equation}
\end{notation}

\blue{Let us present some ideas of proofs. Let us first explain why we see Theorem \ref{t.quenched_cross} as a consequence of \eqref{e.main2} (although we will also need some further work about arm events). (Similarly, we see \eqref{e.main} as a consequence of \eqref{e.main2}.) In \cite{ahlberg2015quenched} (see Theorem 2.1 therein), the authors prove a martingale estimate inspired by the Efron-Stein inequality that implies that the variance of the quenched probability of a crossing event is bounded by $\E [ \sum_{x \in \eta}  (\text{\textup{Inf}}_x^\eta)^2]$, where $\text{\textup{Inf}}_x^\eta$ is the quenched influence, that is, the $\Prob^\eta$-probability that changing the colour of $x$ modifies the indicator function of the crossing event. In both Bernoulli percolation (see \cite{kesten1987scaling,werner2007lectures,nolin2008near}) and Voronoi percolation (see \cite{scaling_voro}), it is known that in some sense, if we neglect boundary issues, the influence is of the same order as the probability of the $4$-arm event. Moreover, by \eqref{e.main2}, \textbf{the quenched probability of the $4$-arm event is typically of the same order as the annealed probability}. As a result, the variance of the quenched probability of a crossing event at scale $R$ is bounded by $CR^2\alpha_4^{an}(R)^2$. This is exactly Theorem \ref{t.quenched_cross}.}
\medskip

\blue{Let us now focus on \eqref{e.main2} which is the central result of our work.} By Jensen's inequality, $\widetilde{\alpha}_j(r,R) \geq \alpha_j^{an}(r,R)$. Our goal is to prove that the other inequality is true up to a constant. In order to explain the general strategy, we need to introduce an annealed and a quenched notions of pivotal events (that we have used in~\cite{scaling_voro}):
\begin{defi}\label{d.piv}
Let $A$ be an event measurable with respect to the coloured configuration $\omega$ and let $\eta$ be the underlying (non-coloured) point configuration. Also, let $D$ be a bounded Borel subset of the plane.
\bi 
\item The subset $D$ is said quenched-pivotal for $\omega$ and $A$ if there exists $\omega' \in \lbrace -1,1 \rbrace^\eta$ such that $\omega$ and $\omega'$ coincide on $\eta \cap D^c$ and $\un_A(\omega') \neq \un_A(\omega)$. We write $\Piv^q_D(A)$ for the event that $D$ is quenched-pivotal for $A$.
\item The subset $D$ is said annealed-pivotal for some Voronoi percolation configuration $\omega$ and some event $A$ if both $\Pro \left[ A \, | \, \omega \setminus D \right]$ and $\Pro \left[ \neg A \, | \, \omega \setminus D \right]$ are positive. We write $\Piv_D(A)$ for the event that $D$ is annealed-pivotal for $A$.
\ei
Note that we have $\Pro \left[ \Piv^q_D(A) \setminus \Piv_D(A) \right]=0$ for any $A$ and $D$ as above.
\end{defi}

\blue{Let us recall that we want to prove that $\widetilde{\alpha}_j(r,R)$ is of the same order as $\alpha_j^{an}(r,R)$. Let us first observe that
\[
\Var \left( \Prob^\eta \left[ \arm_j(r,R) \right] \right)=\widetilde{\alpha}_j(r,R)^2-\alpha_j^{an}(r,R)^2 \, .
\]
As in~\cite{ahlberg2015quenched}, we will use a martingale method inspired by the Efron-Stein inequality in order to bound $\Var \left( \Prob^\eta \left[ \arm_j(r,R) \right] \right)$. The difference with~\cite{ahlberg2015quenched} is that we will prove an estimate that also holds for non-monotonic events. This estimate is Proposition~\ref{p.martingale} and implies that\footnote{\blue{Let us note that it is crucial for us to use an estimate that involves a sum of \textbf{squares} of pivotal probabilities, i.e.\ an $\ell^2$ sum. On the contrary, many variance inequalities, such as for instance the OSSS inequality, involve an $\ell^1$ sum of pivotal probabilities.}}
\[
\Var \left( \Prob^\eta \left[ \arm_j(r,R) \right] \right) \leq \sum_S \E \left[ \Prob^\eta \left[ \Piv_S(\arm_j(r,R)) \right]^2 \right] \, ,
\]
where $S$ ranges (for instance) over all the $1 \times 1$ squares of the grid $\Z^2$.}

We will then need to estimate the quantities $\E \left[ \Prob^\eta \left[ \Piv_S(\arm_j(r,R)) \right]^2 \right]$. To this purpose, we will rely on several estimates from~\cite{scaling_voro}. In Section~\ref{s.asymp}, we will use these estimates in order to prove that
\[
\sum_S \E \left[ \Prob^\eta \left[ \Piv_S(\arm_j(r,R)) \right]^2 \right] \leq O(1) r^{-\epsilon} \, \widetilde{\alpha}_j(r,R)^2
\]
for some $\epsilon >0$. As a result,
\[
\widetilde{\alpha}_j(r,R)^2-\alpha_j^{an}(r,R)^2 \leq \grandO{1} r^{-\epsilon} \, \widetilde{\alpha}_j(r,R)^2
\]
and there exists $r_0>0$ such that, if $r > r_0$, then
\[
\widetilde{\alpha}_j(r,R)^2 \leq 2 \alpha_j^{an}(r,R)^2  \, .
\]
\blue{We will thus obtain the desired result for $r$ sufficiently large and we will conclude that the result holds for every $r$ thanks to the quasi-multiplicativity property of arm-events (see Proposition~\ref{p.QM}).}
\medskip

Let us end this part on the strategy of proofs by the following remark: As mentioned above, in order to prove our main results, we will have to prove estimates on the probabilities of arm events. To this purpose, our strategy will often consist in: i) defining a ``good'' event $G(r,R)$ and then ii) using the trivial bound:
\[
\alpha_j^{an}(r,R) \leq \Pro \left[ \arm_j(r,R) \cond G(r,R) \right] + \Pro \left[ \neg G(r,R) \right] \, .
\]
Of course, we will define $G(r,R)$ in such a way that it is easier to study $\arm_j(r,R)$ under $\Pro \left[ \cdot \, | \, G(r,R) \right]$ than under $\Pro$. The problem here is that we will have estimates of the kind:
\[
\Pro \left[ \neg G(r,R) \right] \leq \varepsilon_1(r)
\]
and
\[
\Pro \left[ \arm_j(r,R) \right] \leq \varepsilon_2(R/r)
\]
for some functions $\varepsilon_1$ and $\varepsilon_2$ that go to $0$ at infinity. So this strategy is not useful at all when $R/r$ is extremely large compared to $r$. To overcome this difficulty, our strategy will often be to fix some $M \gg 1$, prove estimates on quantities of the form $\alpha_j^{an}(\rho,\rho M)$ for any $\rho \geq M$, and then deduce estimates that hold for $\alpha_j^{an}(r,R)$ for any $r \leq R$ by using the quasi-multiplicativity property Proposition \ref{p.QM}. See in particular the proofs of Propositions~\ref{p.alpha_4_geq} and~\ref{p.alpha_4_from_SS}. Note that this strategy is close to the strategy from~\cite{lawler2002onearm} and~\cite{Smirnov2001critical} where the authors compute the arm exponents for critical percolation on the triangular lattice by estimating the probabilities of non-degenerate arm events and then deducing the result for all arm-events by using the quasi-multiplicativity property.

\begin{notation} Let us end this section by some general notations that we will use all along the paper:
\bi 
\item We write $B_R=[-R,R]^2$ and we write $A(r,R)$ for the annulus $[-R,R]^2 \setminus (-r,r)^2$. Also, for every $y \in \R^2$, we write $B_r(y) = y + B_r$ and $A(y;r,R) = y + A(r,R)$.
\item A quad $Q$ is a topological rectangle in the plane with two distinguished opposite sides. Also, a crossing of $Q$ is a black path included in $Q$ that connects the distinguished sides. The event that $Q$ is crossed is written $\cross(Q)$.
\item We use the following notations: \textit{(a)} $\grandO{1}$ is a positive bounded function, \textit{(b)} $\Omega(1)$ is a positive function bounded away from $0$ and \textit{(c)} if $f$ and $g$ are two non-negative functions, then $f \asymp g$ means $\Omega(1) f \leq g \leq \grandO{1} f$.
\ei
\end{notation}

\paragraph{Acknowledgments:} I would like to thank Christophe Garban for many helpful discussions and for his comments on earlier versions of the manuscript. I would also like to thank Vincent Tassion for fruitful discussions and for having welcomed me in Zürich several times. Finally, I wish to thank an anonymous referee for many helpful comments.

\section{The martingale method}\label{s.mart}

In this section, we follow the ideas \red{from} Section~$2$ of~\cite{ahlberg2015quenched} where the authors use a martingale method \red{inspired by the Efron-Stein inequality}. \red{In particular, they prove a general estimate by discovering the colours of the points of the Poisson process one after the other}. In the present paper, we will rather discover the boxes of a grid (i.e.\ at each step we will discover all the points of the Poisson process that belong to \red{some} box). Remember the definition of pivotal events from Definition~\ref{d.piv}.
\begin{prop}\label{p.martingale}
Let $\rho > 0$, \red{let $E \subseteq \Omega$ be a measurable set}, and let $(S_m^\rho)_{m \in \N}$ be an enumeration of the $\rho \times \rho$ squares of the grid $\rho\Z^2$. Then,
\[
\Var \left( \Prob^\eta \left[ E \right] \right) \leq \sum_{m \in \N} \E \left[ \Prob^\eta \left[ \Piv_{S^\rho_m}(E) \right]^2 \right] \, .
\]
\end{prop}

\begin{proof}
The proof is very similar to the proof of Theorem~$2.1$ in~\cite{ahlberg2015quenched}. We use the following notations:
\[
q^\eta = \Prob^\eta \left[ E \right] \, ;
\]
\[
\forall m \in \N \cup \lbrace -1 \rbrace, \hspace{1em} q_m = \Pro \left[ E \cond \eta \cap (\cup_{k=0}^m S_k^\rho) \right] = \E \left[ q^\eta \cond \eta \cap \left( \cup_{k=0}^m S_k^\rho \right) \right] \, .
\]
Note that $(q_m)_m$ is a bounded martingale that converges in $L^2$ to $q^\eta$. Note also that $q_{-1} = \E \left[ q^\eta \right]$. Hence we have:
\[
\Var \left( q^\eta \right) = \underset{M \rightarrow +\infty}{\text{lim}} \Var \left( \sum_{m = 0}^M q_m - q_{m-1} \right) = \sum_{m \in \N} \Var \left( q_m - q_{m-1} \right) \, .
\]
It is thus sufficient to prove that for all $m \in \N$ we have
\begin{equation}\label{e.varleqpiv}
\Var \left( q_m - q_{m-1} \right) \leq \E \left[ \Prob^\eta \left[ \Piv_{S_m^\rho}(E) \right]^2 \right].
\end{equation}
To this purpose, let \red{$\eta^- = \eta \setminus S_m^\rho$}, let $q^{\eta^-}=\Prob^{\eta^-}[E]$ \blue{(more rigorously, $q^{\eta^-}=\Prob^{\eta^-}[\{-1,1\}^{\eta^-} \cap E]$),\footnote{\red{Note that $\Prob^{\eta^-}[E]$ is different from $\Pro [ E |  \eta^-]$ in general. Under $\Prob^{\eta^-}$, there is no point in $S_m^\rho$.}}} and let us prove the following:
\begin{equation}\label{e.varleqexp}
\Var \left( q_m - q_{m-1} \right) \leq \E \left[ (q^\eta - q^{\eta^-})^2 \right].
\end{equation}
\begin{proof}[Proof of~\eqref{e.varleqexp}]
We follow the proof of Lemma~2.4 in~\cite{ahlberg2015quenched} where the authors use the conditional variance formula.\footnote{The conditional variance is $\Var(X \, | \, Y) = \E \left[ X^2 \, | \, Y \right] - \E \left[ X \, | \, Y \right]^2 = \E \left[ (X-\E \left[ X \, | \, Y \right])^2 \, | \, Y \right]$. The conditional variance formula is $\Var(X)=\Var(\E \left[X \, | \, Y \right]) + \E \left[ \Var(X \, | \, Y) \right]$.} Since $\E \left[ q_m - q_{m-1} \cond \eta \cap (\cup_{k=0}^{m-1} S^\rho_k) \right] = 0$, this formula implies that:
\[
\Var \left( q_m - q_{m-1} \right) = \E \left[ \Var \left( q_m \cond \eta \cap (\cup_{k=0}^{m-1} S^\rho_k) \right) \right] \, .
\]
By using the fact that $(q^{\eta^-},\cup_{k=0}^{m-1} \eta \cap S^\rho_k)$ is independent of $\eta \cap S^\rho_m$, we obtain that 
\[\E \left[ q^{\eta^-} \cond  \eta \cap (\cup_{k=0}^{m-1} S^\rho_k) \right] = \E \left[ q^{\eta^-} \cond  \eta \cap (\cup_{k=0}^{m} S^\rho_k) \right] \, ,
\]
hence $\Var \left( q_m \cond \eta \cap (\cup_{k=0}^{m-1} S^\rho_k) \right)$ equals
\begin{align*}
& \Var \left( \E \left[ q^\eta \, | \, \eta \cap (\cup_{k=0}^m S^\rho_k) \right] \cond \eta \cap (\cup_{k=0}^{m-1} S^\rho_k) \right)\\
& = \Var \left( \E \left[ q^\eta \, | \, \eta \cap (\cup_{k=0}^m S^\rho_k) \right] - \E \left[ q^{\eta^-} \, | \, \eta \cap (\cup_{k=0}^{m-1} S^\rho_k) \right] \cond \eta \cap (\cup_{k=0}^{m-1} S^\rho_k) \right)\\
& = \Var \left( \E \left[ q^\eta - q^{\eta^-} \, | \, \eta \cap (\cup_{k=0}^m S^\rho_k) \right] \cond \eta \cap (\cup_{k=0}^{m-1} S^\rho_k) \right)\\
& \leq \E \left[ \E \left[ q^\eta - q^{\eta^-} \, | \, \eta \cap (\cup_{k=0}^{\red{m}} S^\rho_k) \right]^2 \cond \eta \cap (\cup_{k=0}^{m-1} S^\rho_k) \right] \text{ since } \Var ( \cdot ) \leq \E \left[ \cdot^2 \right] \\
& \leq \E \left[ (q^\eta - q^{\eta^-})^2 \cond \eta \cap (\cup_{k=0}^{m-1} S^\rho_k) \right] \text{ by Jensen's inequality} \, .
\end{align*}
This ends the proof.
\end{proof}
\blue{As a result,
\[
\Var(q^\eta) \leq \sum_{m \in \N} \E \left[ (q^\eta - q^{\eta^-})^2 \right]
\]
so it is now sufficient to prove that a.s.\ we have $|q^\eta - q^{\eta^-}| \leq \Prob^\eta \left[ \Piv_{S_m^\rho}(E) \right]$. \red{To show this, we let $\omega^- = \omega \setminus S^\rho_m$, and we observe that $\Pro$-a.s.\ we have
\[
|\un_{\omega^- \in E} - \un_{\omega \in E}| \leq \un_{\omega \in \Piv_{S_m^\rho}(E)}.
\]
If we condition on $\eta$, then the law of $\omega^-$ is $\Prob^{\eta^-}$. As a result, $\Pro$-a.s.\ we have
\[
\left|q^\eta - q^{\eta^-} \right| = \left| \Prob^\eta \left[ E \right] - \Prob^{\eta^-} \left[ E \right] \right| \leq \E \left[ \left|\un_{\omega^- \in E} - \un_{\omega \in E} \right| \cond \eta \right] \le \Prob^\eta \left[ \Piv_{S_m^\rho}(E) \right].
\]
This ends the proof.}}
\end{proof}

\section{First estimates on arm and pivotal events}\label{s.estimates}

\subsection{Arm-events and pivotal events}\label{ss.a_and_piv}

As one can see in Proposition~\ref{p.martingale}, one way to estimate $\Var \left( \Prob^\eta \left[ \arm_j(r,R) \right] \right)$ is to find upper bounds for the quantities $\E \left[ \Prob^\eta \left[ \Piv_S(\arm_j(r,R)) \right]^2 \right]$. In~\cite{scaling_voro} (see in particular Appendix~D therein), we have proved such upper bounds in terms of the quantities $\widetilde{\alpha}_4(\cdot,\cdot)$ and $\widetilde{\alpha}_j(\cdot,\cdot)$.

Below, we list the results from~\cite{scaling_voro} about the quantities $\widetilde{\alpha}_j(r,R)$ and $\alpha^{an}_j(r,R)$ that we use in the present paper.
%A lot of estimates of the quantities $\alpha^{an}_j(r,R)$ were very important in~\cite{scaling_voro} in order to prove that the annealed scaling relations (see Theorem~\ref{t.scaling_rel} of the present paper). The estimates of $\widetilde{\alpha}_j(r,R)$ were not used in~\cite{scaling_voro} but we explained in Appendix~D of this paper how to extend so
Let us first state a polynomial decay property (see~(1.1) and~(D.3) in~\cite{scaling_voro}): For every $j \in \N^*$, there exists $C=C(j) \in [1,+\infty)$ such that, for every $1 \leq r \leq R <+\infty$,
\begin{equation}\label{e.poly}
\frac{1}{C} \left( \frac{r}{R} \right)^{C} \leq \alpha^{an}_j(r,R) \leq \widetilde{\alpha}_j(r,R) \leq C \left( \frac{r}{R} \right)^{1/C} \, .
\end{equation}
The following quasi-multiplicativity property will be crucial in the present work.
\begin{prop}[Propositions~1.6 and~D.1 of~\cite{scaling_voro}]\label{p.QM}
Let $j \in \N^*$. There exists a constant $C=C(j) \in [1,+\infty)$ such that, for every $1 \leq r_1 \leq r_2 \leq r_3$,
\[
\frac{1}{C} \, \alpha^{an}_j(r_1,r_3) \leq \alpha^{an}_j(r_1,r_2) \, \alpha^{an}_j(r_2,r_3) \leq C \, \alpha^{an}_j(r_1,r_3) \, .
\]
and
\[
\frac{1}{C} \, \widetilde{\alpha}_j(r_1,r_3) \leq \widetilde{\alpha}_j(r_1,r_2) \, \widetilde{\alpha}_j(r_2,r_3) \leq C \, \widetilde{\alpha}_j(r_1,r_3) \, .
\]
\end{prop}
We have the following estimates on $4$-arm events:
\begin{prop}[Corollary D.11 of~\cite{scaling_voro}]\label{p.alpha_4}
There exists $\epsilon > 0$ such that, for every $R \in [1,+\infty)$,
\[
\alpha^{an}_4(R) \leq \widetilde{\alpha}_4(R) \leq \frac{1}{\epsilon} R^{-(1+\epsilon)} \, .
\]
\end{prop}

We will prove a multiscale version of Proposition~\ref{p.alpha_4} in Section~\ref{s.other}.

\begin{prop}[Proposition~1.13 of~\cite{scaling_voro}]\label{p.alpha_4_2}
There exists $\epsilon > 0$ such that, for every $1 \leq r \leq R < +\infty$,
\[
\widetilde{\alpha}_4(r,R) \geq \alpha^{an}_4(r,R) \geq \epsilon \left( \frac{r}{R} \right)^{2-\epsilon} \, .
\]
\end{prop}

We will improve Proposition~\ref{p.alpha_4_2} in Section~\ref{s.strict}.
\medskip

Let us write $\arm_j^+(r,R)$ for the $j$-arm event in the half-plane, whose definition is the same as the definition of $\arm_j(r,R)$ except that we ask that the arms live in the (upper, say) half-plane. We also write $\alpha_j^{an,+}(r,R)=\Pro \left[ \arm_j^+(r,R) \right]$ and $\widetilde{\alpha}_j^+(r,R) = \sqrt{ \E \left[ \Prob^\eta \left[ \arm_j^+(r,R) \right]^2 \right]}$. We have the following:
\begin{prop}[Proposition~2.7 of~\cite{scaling_voro}]\label{p.universal} The computation of the universal arm exponents holds for annealed Voronoi percolation: Let $1 \leq r \leq R$. Then,
\bi
\item[\textup{i)}] $\alpha^{an,+}_2(r,R) \asymp r/R \,$, hence $\Omega(1) (r/R) \leq \widetilde{\alpha}^+_2(r,R) \leq \grandO{1} (r/R)^{1/2}$ by Jensen's inequality,
\item[\textup{ii)}] $\alpha^{an,+}_3(r,R) \asymp \left( r/R \right)^2 \,$, hence $\Omega(1) (r/R)^2 \leq \widetilde{\alpha}^+_3(r,R) \leq \grandO{1} r/R \,$,
\item[\textup{iii)}] $\alpha^{an}_5(r,R) \asymp \left( r/R \right)^2 \,$, hence $\Omega(1) (r/R)^2 \leq \widetilde{\alpha}_5(r,R) \leq \grandO {1} r/R \,$. 
\ei
\end{prop}

\begin{rem}
Thanks to~\eqref{e.main2} (from Theorem~\ref{t.quenched_arm}), we will be able to deduce from Proposition~\ref{p.universal} that
\begin{equation}\label{e.universal_bis}
\widetilde{\alpha}^+_2(r,R) \asymp r/R, \, \widetilde{\alpha}^+_3(r,R) \asymp (r/R)^2 \text{ and } \widetilde{\alpha}_5(r,R) \asymp (r/R)^2 \, .
\end{equation}
However, in order to prove~\eqref{e.main2}, we will only be able to rely on the weaker estimates from Proposition~\ref{p.universal}. The reason why we have not managed to prove~\eqref{e.universal_bis} without relying on~\eqref{e.main2} is that the computation of these universal exponents uses crucially the translation invariance properties of the \textbf{annealed} model.
\end{rem}

In Appendix~D.2 of~\cite{scaling_voro}, we have proved upper bounds for the quantities
\begin{equation}\label{e.ghftf}
\E \left[ \Prob^\eta \left[ \Piv_S( \arm_j(r,R) ) \right]^2 \right] \, ,
\end{equation}
where $S$ is a square included in the annulus $A(r,R)$. Here, we state five lemmas (that are consequences of the results from~\cite{scaling_voro} or that can be proved by using methods from~\cite{scaling_voro}, see the sketch of proof below) that give upper bounds for~\eqref{e.ghftf} when $S$ is respectively in the bulk of $A(r,R)$, near the outer boundary of this annulus, in the unbounded connected component of $\R^2 \setminus A(r,R)$, near the inner boundary of this annulus, and in the bounded component of $\R^2 \setminus A(r,R)$.
\medskip

Let $y$ be a point of the plane, let $\rho \geq 1$, let $S = B_\rho(y)$, and let $r,R$ be such that $\rho \leq r/10$ and $r \leq R/2$. Also, let $j \in \N^*$.

\begin{lem}\label{l.piv_arm_1}
Let $y$, $\rho$, $r$, $R$ and $S=B_{\rho}(y)$ be as above. Assume that $S \subseteq A(2r,R/2)$ and let $d \geq r$ be the distance between $y$ and $0$. Then,
\[
\E \left[ \Prob^\eta \left[ \Piv_S( \arm_j(r,R) ) \right]^2 \right] \leq \grandO{1} \left( \widetilde{\alpha}_j(r,R) \, \widetilde{\alpha}_4(\rho,d)  \right)^2 \, .
\]
\end{lem}

If $S \cap A(R/2,R) \neq \emptyset$ then we use the following notations: Let $d_0=d_0(S)$ be the distance between $S$ and the closest side of $B_R$ and let $y_0$ be the orthogonal projection of $y$ on this side. Also, let $d_1=d_1(S) \geq d_0$ be the distance between $y_0$ and the closest corner of $B_R$. Write $\arm^{++}_j(\cdot,\cdot)$ for the $j$-arm event in the quarter plane and let $\widetilde{\alpha}^{++}_j(\cdot,\cdot) := \sqrt{\E \left[ \Prob^\eta \left[ \arm^{++}_j(\cdot,\cdot) \right]^2 \right]}$. We have the following:

\begin{lem}\label{l.piv_arm_2}
Let $y$, $\rho$, $r$, $R$ and $S=B_{\rho}(y)$ be as above. Assume that $S \cap A(R/2,R) \neq \emptyset$. Remember that $\rho \leq r/10$ and $r \leq R/2$. Then,
\[
\E \left[ \Prob^\eta \left[ \Piv_S( \arm_j(r,R) ) \right]^2 \right] \leq \grandO{1} \left( \widetilde{\alpha}_j(r,R) \, \widetilde{\alpha}^{++}_3(d_1+\rho,R) \, \widetilde{\alpha}^{+}_3(d_0+\rho,d_1) \, \widetilde{\alpha}_4(\rho,d_0) \right)^2 \, . 
\]
\end{lem}

The following lemma roughly says that, if we want to use our bounds to estimate the sum
\begin{equation}\label{e.sum_infinite_finite}
\sum_{S \text{ square of the grid } 2\rho\Z^2} \E \left[ \Prob^\eta \left[ \Piv_{S}( \arm_j(r,R) ) \right]^2 \right] \, ,
\end{equation}
and if we forget the terms corresponding to the squares $S$ that are in the unbounded component of $\R^2 \setminus A(r,R)$, then this does not change the order of the estimate. \blue{(Here and below, a ``square of the grid $2\rho\Z^2$'' is a unit square of this grid, i.e.\ a $2\rho \times 2\rho$ square.) More precisely, given a $2\rho \times 2\rho$ square $S$ that intersects $\partial B_R$ (in particular, such a square satisfies the hypothesis of Lemma \ref{l.piv_arm_2}), we prove that the sum \eqref{e.sum_infinite_finite} restricted to the (infinite) family of squares $S'$ of the grid $2\rho\Z^2$ that satisfy the property ``$S'$ does not intersect $B_R$ and, among all squares of the grid $2\rho\Z^2$ that intersect $\partial B_R$, $S$ is the closest'' is less than or equal to the bound from Lemma \ref{l.piv_arm_2} for the \textbf{single} square $S$.}

\begin{figure}[!h]
\begin{center}
\includegraphics[scale=0.5]{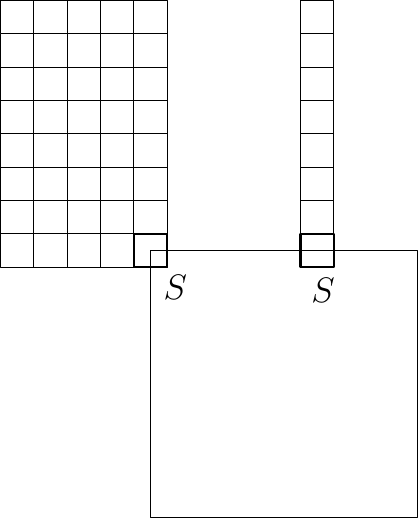}
\end{center}
\caption{\blue{A square $S$ and the squares $S' \in \setS$ from Lemma \ref{l.piv_arm_3} in two cases.}}
\end{figure}

\begin{lem}\label{l.piv_arm_3}
Let $\rho \geq 1$ and let $r,R$ be such that $\rho \leq r/100$ and $r \leq R/2$. Also, let $S$ be a square of the grid $2\rho\Z^2$ that intersects $\partial B_R$. Moreover, let $\setS$ be the set of all squares $S'$ of the grid $2\rho \Z^2$ that do not intersect $B_R$ and are such that $S$ is the argmin of $\dist(S'',S')$ where $S''$ ranges over the set of squares of the grid $2\rho\Z^2$ that intersect $\partial B_R$. Then,
\[
\sum_{S' \in \setS} \E \left[ \Prob^\eta \left[ \Piv_{S'}( \arm_j(r,R) ) \right]^2 \right] \leq \grandO{1} \left( \widetilde{\alpha}_j(r,R) \, \widetilde{\alpha}^{++}_3(d_1+\rho,R) \, \widetilde{\alpha}^{+}_3(d_0+\rho,d_1) \, \widetilde{\alpha}_4(\rho,d_0) \right)^2 \, ,
\]
where $d_0=d_0(S)=0$ and $d_1=d_1(S)$.
\end{lem}

Let us now study the quantity $\E \left[ \Prob^\eta \left[ \Piv_{S}( \arm_j(r,R) ) \right]^2 \right]$ when $S$ is at distance less than $2r$ from $0$. If $S \cap A(r,2r) \neq \emptyset$, we use the following notations: Let $d_0=d_0(S)$ be the distance between $S$ and the closest side of $B_r$ and let $y_0$ be the orthogonal projection of $y$ on this side. Also, let $d_1=d_1(S)$ be the distance between $y_0$ and the closest corner of $B_r$. Write $\arm^{(++)^c}_j(\cdot,\cdot)$ for the $j$-arm event in the plane without the quarter plane and let $\widetilde{\alpha}^{(++)^c}_j(\cdot,\cdot) = \sqrt{\E \left[ \Prob^\eta \left[ \arm^{(++)^c}_j(\cdot,\cdot) \right]^2 \right]}$. 

\begin{lem}\label{l.piv_arm_4}
Let $y$, $\rho$, $r$, $R$ and $S=B_{\rho}(y)$ be as above. Assume that $S \cap A(r,2r) \neq \emptyset$. Remember that $\rho \leq r/10$ and $r \leq R/2$. If $d_1\geq d_0$, then
\[
\E \left[ \Prob^\eta \left[ \Piv_S( \arm_j(r,R) ) \right]^2 \right] \leq \grandO{1} \left( \widetilde{\alpha}_j(r,R) \, \widetilde{\alpha}^{(++)^c}_3(d_1+\rho,r) \, \widetilde{\alpha}^{+}_3(d_0+\rho,d_1) \, \widetilde{\alpha}_4(\rho,d_0) \right)^2 \, . 
\]
If we rather have $d_1 \leq d_0$, then
\[
\E \left[ \Prob^\eta \left[ \Piv_S( \arm_j(r,R) ) \right]^2 \right] \leq \grandO{1} \left( \widetilde{\alpha}_j(r,R) \, \widetilde{\alpha}^{(++)^c}_3(d_0+\rho,r) \, \widetilde{\alpha}_4(\rho,d_0) \right)^2 \, . 
\]
\end{lem}

\begin{lem}\label{l.piv_arm_5}
Let $\rho \geq 1$ and let $r,R$ be such that $\rho \leq r/100$ and $r \leq R/2$. Also, let $S$ be a square of the grid $2\rho\Z^2$ that intersects $\partial B_r$. Moreover, let $\setS$ be the set of all squares $S'$ of the grid $2\rho \Z^2$ that are included in $B_r$ and are such that $S$ is the argmin of $\dist(S'',S)$ where $S''$ spans over the set of squares of the grid $2\rho\Z^2$ that intersects $\partial B_r$. Then,
\[
\sum_{S' \in \setS} \E \left[ \Prob^\eta \left[ \Piv_{S'}( \arm_j(r,R) ) \right]^2 \right] \leq \grandO{1} \left( \widetilde{\alpha}_j(r,R) \, \widetilde{\alpha}^{(++)^c}_3(d_1+\rho,r) \, \widetilde{\alpha}^{+}_3(d_0+\rho,d_1) \, \widetilde{\alpha}_4(\rho,d_0) \right)^2 \, ,
\]
where $d_0=d_0(S)=0$ and $d_1=d_1(S)$.
\end{lem}

\begin{proof}[Proof of Lemmas~\ref{l.piv_arm_1} to~\ref{l.piv_arm_5}] In Section~4.3 of~\cite{scaling_voro}, we have proved analogous estimates for the quantities $\Pro \left[ \Piv_S(\arm_j(1,R) \right]$. Moreover, Lemma~D.13 of~\cite{scaling_voro} gives estimates on the quantities $\E \left[ \Prob^\eta \left[ \Piv_S(\arm_j(r,R) \right]^2 \right]$ when $S$ is in the ``bulk'' of $B_R$. In particular, this lemma implies Lemma~\ref{l.piv_arm_1}. The proof of Lemmas~\ref{l.piv_arm_2} and~\ref{l.piv_arm_4} is very similar except that we have to take care about boundary issues. The way to adapt the proofs in the case where $S$ is in the bulk to the case where $S$ is close to the boundary is the same as in Section~4.3 of~\cite{scaling_voro}, so we leave the details to the reader. Similarly, the way we deduce Lemmas~\ref{l.piv_arm_3} and~\ref{l.piv_arm_5} from respectively Lemmas~\ref{l.piv_arm_2} and~\ref{l.piv_arm_4} is the same as for the analogous results from Section~4.3 of~\cite{scaling_voro}.
\end{proof}

\subsection{The ``good'' events}\label{ss.good}

\blue{The results from this subsection are not used in the proof of \eqref{e.main} from Theorem \ref{t.quenched_arm}. So the reader who is mainly interested in this result can skip this subsection.}

Since we study a model in random environment, it is important to have estimates on some ``good'' events measurable with respect to $\eta$. The definitions and the estimates that we state in this section are from~\cite{scaling_voro}. We first define the ``dense'' events that help us to have spatial independence properties.
\begin{defi}
If $\delta \in (0,1)$ and $D$ is a bounded Borel subset of the plane, we write $\dense_\delta(D)$ for the event that, for every $u \in D$, there exists $x \in \eta \cap D$ such that $||x-u||_2 < \delta \cdot \diam(D)$.
\end{defi}
\begin{lem}[e.g. Lemma~2.11 of~\cite{scaling_voro}]\label{l.dense}
Let $R \geq 1$ and $\delta \in (0,1)$. We have
\[
\Pro \left[ \dense_\delta(B_R) \right] \geq 1-\grandO{1} \, \delta^{-2} \, \exp \left( -\frac{(\delta \cdot R)^2}{2} \right) \, .
\]
\end{lem}

\blue{We will often use the properties of the ``dense'' events without providing many details so let us explain here how we will use these events: let $D_1,D_2 \subseteq \R^2$ be two disjoint sets and let $A_i$, $i\in \{1,2\}$ be an event that depends only on the colours of the points in $D_i$ (by ``points'' we mean ``point of the plane'' and not ``point of $\eta$''; for instance, $\cross(R,R)$ depends only on the colours of the points in $[-R,R]^2$). Moreover, assume that two ``dense'' events $\dense_i=\dense_{\delta_i}(D_i)$, $i \in \{1,2\}$ hold and that under $\dense_1$ (resp. $\dense_2$) no Voronoi cell of a point $x \in \eta \cap D_2$ can intersect $D_1$ (resp. no Voronoi cell of a point $x \in \eta \cap D_1$ can intersect $D_2$). Then, $A_1 \cap \dense_1$ is independent of $A_2 \cap \dense_2$. Moreover, if $\eta \in \dense_1 \cap \dense_2$ then $A_1$ is independent of $A_2$ under $\Prob^\eta$.}
\medskip

We now define some sets of quads and state a result from~\cite{scaling_voro} that roughly says that, with high probability, the quenched crossing probabilities of all the quads in these sets are non-negligible. The main tool in the proof of this result was the quenched box-crossing result of~\cite{ahlberg2015quenched}.

\begin{defi}\label{d.a_lot_of_quads}
Let $D$ be a bounded subset of the plane and let $\delta > 0$. We denote by $\mathcal{Q}'_\delta(D)$ the set of all quads $Q \subseteq D$ which are drawn on the grid $(\delta \, \diam(D)) \cdot \Z^2$ (i.e. whose sides are included in the edges of $(\delta \, \diam(D)) \cdot \Z^2$ and whose corners are vertices of $(\delta \, \diam(D)) \cdot \Z^2$). Also, we denote by $\mathcal{Q}_\delta(D)$ the set of all quads $Q \subseteq D$ such that there exists a quad $Q' \in \mathcal{Q}'_\delta(D)$ satisfying $\cross(Q') \subseteq \cross(Q)$. \red{Note that these sets are empty if $\delta > 1$.}

Moreover, we let $\widetilde{\mathcal{Q}}'_\delta(D) \subseteq \mathcal{Q}'_\delta(D)$ be the set of all quads $Q \subseteq D$ such that there exists $k \in \N$ such that $Q$ is  drawn on the grid $(2^k \, \delta \, \diam(D)) \cdot \Z^2$ and the length of each side of $Q$ is less than $100 \cdot 2^k \, \delta \, \diam(D)$. Also, we write $\widetilde{\mathcal{Q}}_\delta(D)$ for the set of all quads $Q \subseteq D$ such that there exists a quad $Q' \in \widetilde{\mathcal{Q}}'_\delta(D)$ satisfying $\cross(Q') \subseteq \cross(Q)$.
\end{defi}

\blue{Thus, $\widetilde{\mathcal{Q}}'_\delta(D)$ consists of the quads in $D$ that are ``not too long'' and whose opposite sides are ``not too close to each other''.}

\begin{prop}[Proposition~3.2 of~\cite{scaling_voro}\footnote{\red{This proposition is stated for $\delta \in (0,1)$ but the proof is the same if $\delta = 1$ and the result is trivial if $\delta > 1$ because the set $\widetilde{\mathcal{Q}}_{\delta}(D)$ is empty in this case. }}]\label{p.a_lot_of_quads}
Let $\delta,\gamma \in (0,+\infty)$. There exist an absolute constant $C<+\infty$ and a constant $\widetilde{c} = \widetilde{c}(\gamma) \in (0,1)$ that does not depend on $\delta$ such that, for every bounded subset of the plane $D$ that satisfies $\diam(D) \geq \delta^{-2}/100$, we have
\[
\Pro \left[ \widetilde{\qbc}^\gamma_\delta(D) \right] \geq 1 - C \diam(D)^{-\gamma} \, ,
\]
where
\[
\widetilde{\qbc}^\gamma_\delta(D) = \left\lbrace \forall Q \in \widetilde{\mathcal{Q}}_{\delta}(D), \, \Prob^\eta \left[ \cross(Q) \right] \geq \widetilde{c}(\gamma) \right\rbrace \, .
\]
(The notation $\qbc$ means ``Quenched Box Crossings''.)
\end{prop}

\red{For every $\gamma > 0$, we fix a constant $\widetilde{c}(\gamma)$ as in Proposition~\ref{p.a_lot_of_quads}.}

\begin{rem}\label{r.a_lot_of_quads}
Note that, by gluing arguments,\footnote{\blue{More precisely, we use  that if all the (vertical and horizontal) $2\delta\diam(D) \times \delta \diam(D)$ rectangles \red{and all the $\delta\diam(D) \times \delta \diam(D)$ squares} of the grid $\delta\diam(D)\Z^2$ \red{included in $D$} are crossed, then this is also the case of all the quads $Q \in \mathcal{Q}_\delta(D)$.}} there exists an absolute constant $C_1 \in (0,+\infty)$ such that, if $c=c(\delta,\gamma):=\widetilde{c}(\gamma)^{C_1\delta^{-2}}$, then
\[
\widetilde{\qbc}^\gamma_\delta(D)  \subseteq \bigcap_{k \in \N} \qbc^\gamma_{2^k\delta}(D) \subseteq \qbc^\gamma_{\delta}(D) \, ,
\]
where
\[
\qbc^\gamma_\delta(D) = \left\lbrace \forall Q \in \mathcal{Q}_{\delta}(D), \, \Prob^\eta \left[ \cross(Q) \right] \geq c(\delta,\gamma) \right\rbrace \, .
\]
\red{For every $\gamma,\delta > 0$, we fix a constant $c(\delta,\gamma)$ as above (i.e.\ we fix a constant $C_1$ as above).}

\red{As a result, Proposition \ref{p.a_lot_of_quads} implies that} there exists an absolute constant $C<+\infty$ such that, for every $\delta,\gamma \in (0,+\infty)$, and every bounded subset of the plane $D$ satisfying $\diam(D) \geq \delta^{-2}/100$, we have
\[
\Pro \left[ \qbc^\gamma_\delta(D) \right] \geq 1 - C \diam(D)^{-\gamma} \, .
\] 
\end{rem}

%The last event that we study is the event that interfaces are well-separated:  Let $\delta \in (0,1)$, let $1 \leq r \leq R$, and let $\beta_1, \cdots, \beta_k$ be the interfaces from $\partial B_r$ to $\partial B_R$. Also, let $z_i^{ext}$ (respectively $z_i^{int}$) denote the endpoint on $\partial B_R$ (respectively on $\partial B_r$) of $\beta_i$, and let $s^{ext}(r,R)$ (resp. $s^{int}(r,R)$) be the least distance between $z_i^{ext}$ (resp. $z_i^{int}$) and $\cup_{j \neq i} \beta_j$. Let $\gi^{ext}_\delta (R)$ (for ``Good Interfaces'') be the event that there does not exist $y \in \partial B_R$ such that the $3$-arm event in $A(y;10 \, \delta \, R, R/4) \cap B_R$ holds. Note that, if $r \leq 3R/4$:
%\[
%\gi^{ext}_\delta (R) \subseteq \lbrace s^{ext}(r,R) \geq 10 \delta R \rbrace \, .
%\]
%Similarly, let $\gi^{int}_\delta (r)$ be the event that there does not exist $y \in \partial B_r$ such that the $3$-arm event in $A(y;10 \, \delta \, r, r/2) \setminus B_r$ holds. Note that, if $R \geq 3r/2$:
%\[
%\gi^{int}_\delta (r) \subseteq \lbrace s^{int}(r,R) \geq 10 \delta r \rbrace \, .
%\]
%We have the following estimate:
%\begin{lem}[Lemma~?? of~\cite{scaling_voro}]\label{l.interfaces}
%Let $\delta \in (0,1)$ and let $r,R \geq 100 \, \delta$. There exist absolute constants $C < +\infty$ and $\epsilon > 0$ such that:
%\[
%\Pro \left[ \gi^{ext}_\delta(R) \right] \geq 1- C \, \delta \, ,
%\]
%and:
%\[
%\Pro \left[ \gi^{int}_\delta(r) \right] \geq  1- C \, \delta^{\epsilon} \, .
%\]
%\end{lem}

\subsection{A quenched quasi-multiplicativity property}\label{ss.quenched_prop}

\blue{The results from this subsection are not used in the proof of \eqref{e.main} from Theorem \ref{t.quenched_arm}. Moreover, they are not used in the proof of Theorem \ref{t.strict}. So the reader who is mainly interested in these results can skip this subsection.}

In~\cite{scaling_voro}, we have proved the quasi-multiplicativity property for the quantities $\alpha^{an}_j(r,R)$ and $\widetilde{\alpha}_j(r,R)$ (see Proposition~\ref{p.QM} of the present paper). The proof was rather technical because of the multiple passages from quenched to annealed estimates. The proof of the following property is much easier.
\begin{prop}\label{p.quenched_QM}
For every $\gamma > 0$ and every $j \in \{ 1 \} \cup 2\N^*$, there exists $C=C(\gamma,j) \in [1,+\infty)$ such that, for every $r_0 \in [1,+\infty)$, the following holds with probability larger than $1-C r_0^{-\gamma}$: For every $r_1,r_2,r_3 \in [r_0,+\infty)$ that satisfy $r_1 \leq r_2 \leq r_3$, we have
\begin{equation}\label{e.quenched_QM}
\frac{1}{C} \, \Prob^\eta \left[ \arm_j(r_1,r_3) \right] \leq \Prob^{\eta} \left[ \arm_j(r_1,r_2) \right] \, \Prob^\eta \left[ \arm_j(r_2,r_3) \right] \leq C \, \Prob^\eta \left[ \arm_j(r_1,r_3) \right] \, .
\end{equation}
\end{prop}
\begin{proof}
Fix $\gamma > 0$. We write the proof for $j=4$ since the proof is the same for \red{the other even positive integers $j$} and is simpler for $j=1$. Let $\delta_0 \in (0,1/1000)$, let $A_n(r_0)$ be the annulus $A(5^{n-2}r_0, 5^{n+2} r_0)$, and consider the event
\begin{equation}\label{e.gpgammadelta0}
\gp^\gamma_{\delta_0}(r_0) = \bigcap_{n \geq 0} \dense_{\delta_0} \left( A_n(r_0) \right) \cap \widetilde{\qbc}_{\delta_0}^\gamma \left( A_n(r_0) \right) \, .
\end{equation}
(Where $\gp$ means ``Good Point process''.) If we follow the classical proofs of the quasi-multiplicativity property on non-random lattices, we obtain that~\eqref{e.quenched_QM} holds if $\eta \in \gp^\gamma_{\delta_0}(r_0)$ with $\delta_0$ sufficiently small. Let us be more precise: let $\eta \in \gp^\gamma_{\delta_0}(r_0)$ and let us follow Appendix~A of~\cite{schramm2010quantitative}, where the quasi-multiplicativity property is proved for bond percolation on $\Z^2$ and site percolation on the triangular lattice. All the independence properties that are needed in this appendix hold since we work \blue{under $\Prob^\eta$} and since $\eta \in \bigcap_{n \geq 0} \dense_{\delta_0/100}(A_n(r_0))$. There are three steps in the proof from~\cite{schramm2010quantitative} (which correspond respectively to Lemmas~A.2,~A.3 and~A.4 therein):
\begin{enumerate}
\item \blue{In the first step, the authors prove (by using box-crossing arguments) that there exist $C<+\infty$ and $\epsilon>0$ such that, for every $R \geq 1$, the probability that there exist two interfaces that cross the annulus $A(R,2R)$ and whose endpoints are at distance less than $R\delta$ from each other is less than $C\delta^\epsilon$. Since $\eta \in \bigcap_{n \geq 0} \widetilde{\qbc}_{\delta_0}^\gamma \left( A_n(r_0) \right)$, we can use the same box-crossing arguments to prove that this result holds as soon as $R \geq r_0$ and $\delta \geq \delta_0$ and with constants $C,\eps$ that depend on $\gamma$ (note that here it is important that the constant $\widetilde{c}$ from Proposition~\ref{p.a_lot_of_quads} does not depend on $\delta$).}
\item In the second step, the authors of~\cite{schramm2010quantitative} prove that there exists $\overline{\delta}>0$ such that, for each $\delta > 0$ and each $r,R \geq 1$ satisfying $r \leq R/2$, there exists $a=a(\delta) > 0$ such that we have the following: Let $s(r,R)$ be the minimal distance between the endpoints on $\partial B_R$ of two interfaces that cross $A(r,R)$. If we condition on $\arm_4(r,R) \cap \{ s(r,R) > \delta R \}$, then the probability of $\arm_4(r,4R) \cap \{ s(r,4R) > \overline{\delta} R \}$ is larger than $a$. Since $\eta \in \cap_{n \geq 0} \widetilde{\qbc}_{\delta_0}^\gamma \left( A_n(r_0) \right)$ and since\footnote{See~Remark~\ref{r.a_lot_of_quads}.}
\[
\widetilde{\qbc}_{\delta_0}^{\gamma} \left( A_n(r_0) \right) \subseteq \bigcap_{k \geq 0} \qbc_{2^k  \delta_0}^{\gamma} \left( A_n(r_0) \right) \, ,
\]
we can use the same box-crossing arguments as in~\cite{schramm2010quantitative} to prove that there exists $\overline{\delta}>0$ such that, if $\delta_0 \leq \overline{\delta}$, then this result holds for any $\delta \geq \red{100}\delta_0$ and any $r,R \geq 1$ such that $r \leq R/2$ and $R \geq r_0$ (and with a constant $a=a(\delta)$ that also depends on $\gamma$).

Let us be a little more precise about the adaptation of the proof. Let $r,R$ be such that $r \leq R/2$ and $R \geq r_0$ and assume that $\arm_4(r,R) \cap \{ s(r,R) \geq \delta R\}$ holds (we keep the same notation as in the case of Bernoulli percolation). Moreover, let $k \in \N$ be such that $2^k\delta_0 \leq \delta \leq 2^{k+1} \delta_0$ and $n \in \N$ be such that $5^{n-1}r_0 \leq R \leq 5^nr_0 $. Then, we can use the box-crossing estimates given by $\qbc_{2^k\delta_0}^\gamma(A_n(r_0))$ to extend the four arms with probability larger than some constant $a$ that depends only on $\delta$ and $\gamma$.
\item The third step is a combination of the first two steps \blue{that relies on spatial independence. The proof in our case is the same as in \cite{schramm2010quantitative}.}
\end{enumerate}
Finally, the quasi-multiplicativity property holds for every $r_1,r_2,r_3 \geq r_0$ as soon as $\eta \in \gp^\gamma_{\delta_0}(r_0)$ for some $\delta_0$ sufficiently small, so it only remains to prove that for every $\delta_0$ we have
\[
\Pro \left[ \gp^\gamma_{\delta_0}(r_0) \right] \geq 1-\grandO{1} r_0^{-\gamma} \, ,
\]
where the constants in the $\grandO{1}$ only depend on $\delta_0$ and $\gamma$. This is actually a direct consequence of Lemma~\ref{l.dense} \blue{(or rather of the analogue for annuli instead of squares, but the proof is the same)} and of Proposition~\ref{p.a_lot_of_quads}.
\end{proof}

\begin{rem}
We have stated Proposition~\ref{p.quenched_QM} only for $j=1$ and $j$ even since the proof is less technical in these cases and since we will use this proposition only for $j=4$.
%In order to extend the proof to $j$ odd, one classical technique is to compute the exponent of the $3$-arm event in the half-place in oder to have better separation of arms estimates, but the classical computation of this exponent use translation invariance, which is not possible at the quenched level since the Voronoi tiling is of course not translation invariant. In~\cite{schramm2010quantitative}, other techniques are used in order to extend the quasi-mutliplicativity property to $j$ odd but there is unfortunately a mistake: the new definition of $s(r,R)$ in the proof of Proposition~A.5 therein is not monotonic in general, so we should adapt this proof. One other strategy to extend Proposition~\ref{p.quenched_QM} to $j$ odd is to rather follow the slightly different proof by Kesten~\cite{kesten1987scaling} (see~\cite{nolin2008near}).
\end{rem}

In Section~\ref{s.other}, we will need the following quenched estimate whose proof is roughly the same as that of Proposition~\ref{p.quenched_QM}. We first need to introduce a notation: If $Q$ is a $r \times r$ square and $\alpha > 0$, we let $\alpha Q$ be the square concentric to $Q$ with side length $\alpha r$ and we let $\Circ_\delta(Q)$ be the event that there is a black circuit in the annulus $(1-\delta)Q \setminus (1-2\delta)Q$ and no white circuit in this annulus. \blue{Also, we let $\Circ_\delta^*(Q)$ be the event that there is a white circuit in the annulus $(1-\delta)Q \setminus (1-2\delta)Q$ and no black circuit in this annulus.}
\begin{lem}\label{l.lemme_for_alpha4}
Let $\gamma > 0$. There exists $\widetilde{\delta}=\widetilde{\delta}(\gamma) > 0$ such that, for every $\delta \in (0,\widetilde{\delta}]$, there exist $C =C(\delta,\gamma) <+\infty$, $c=c(\gamma)>0$ and $c'=c'(\delta,\gamma)>0$ such that, for every $r,R \geq 1$, the following holds: Let $Q$ be a $2r \times 2r$ square included in $B_R$ and at distance at least $R/3$ from the sides of $B_R$ and let $x$ denote the center of $Q$. Also, let $X$ be the $\pm 1$ indicator function of $\cross(R,R)$. Then, with probability larger than $1-Cr^{-\gamma}$ we have
\bi 
\item[i)] $\Prob^\eta \left[ \Piv^q_Q(\cross(R,R)) \right] \geq c \Prob^\eta \left[ \arm_4(x;r,R) \right]\,$, where $\arm_4(x;r,R)$ is the $4$-arm event translated by $x$,
\item[ii)] $\Prob^\eta \left[ \Circ_\delta(Q) \right] \geq c'\,$,
\item[iii)] $\Ex^\eta \left[ X \cond \Circ_\delta(Q) \cap \Piv^q_Q(\cross(R,R)) \right] > 1/4 \,$,
\item[iv)] $\Ex^\eta \left[ X \cond \Circ_\delta^*(Q) \cap \Piv^q_Q(\cross(R,R)) \right] < -1/4\,$.
\ei
\end{lem}
\begin{proof}
\blue{Let $\gamma>0$. We write the proof for $Q$ centered at $0$ (i.e.\ $x=0$ and $Q=B_r$) to simplify the notations, and we define $\gp^\gamma_{\cdot}(\cdot)$ as in~\eqref{e.gpgammadelta0}.}
\medskip

\blue{Let $\delta \in (0,1)$ and let $\eta \in \gp^\gamma_{\delta/100}(r)$. Until the last line of the proof, we work under the probability measure $\Prob^\eta$. We first note that Item~ii) holds for some $c'=c'(\delta,\gamma)$ because $\eta \in \gp^\gamma_{\delta/100}(r)$. Let us now study the quenched pivotal event. The event $\Piv^q_Q(\cross(R,R))$ is the event that there are: one black path from a cell whose center $x \in \eta$ belongs to $Q$ to the left side of $[-R,R]^2$, one black path from such a cell to the right side, one white path from such a cell the top side and one white path from such a cell to the bottom side. If we follow the proof of Proposition~\ref{p.quenched_QM} we obtain that Item~i) holds if $\delta$ is less than some constant that depends only on $\gamma$. We even obtain that:
\begin{enumerate}
\item i) holds but with $\Prob^\eta \left[ \arm_4(x;r,R) \right]$ replaced by $\Prob^\eta \left[ \arm_4(x;2r,R) \right]$,
\item let $\mathring{Q}$ be the union of all Voronoi cells whose center belongs to $Q$ and let $s^{int}(\mathring{Q},R)$ denote the minimum distance between the endpoints on the (topological) boundary of $\mathring{Q}$ of two interfaces that cross the topological annulus $[-R,R]^2 \setminus \mathring{Q}$. There exist $C,\eps>0$ (that depend on $\gamma$) such that, if $\delta_1 \ge \delta$, then $\Prob^\eta [s^{int}(\mathring{Q},R) \leq \delta_1 r] \leq C\delta_1^\eps$.   
\end{enumerate}
By using the above, the fact that $s^{int}(\mathring{Q},R)$ is $\Prob^\eta$-independent of $\arm_4(x;2r,R)$, and the fact that $\arm_4(x;2r,R) \supseteq  \Piv_Q^q(\cross(R,R))$, we obtain that
\[
\Prob^\eta \left[ s^{int}(\mathring{Q},R) \geq \delta_1 r \cond \Piv_Q^q(\cross(R,R)) \right] \geq 1 - C\delta_1^\eps \frac{\Prob^\eta [\arm_4(x;2r,R)]}{\Prob^\eta [\Piv_Q^q(\cross(R,R))]} \geq 1 - C\delta_1^\eps/c \, .
\]
Finally, by classical box-crossing arguments, one can prove that there exist $C',\eps'>0$ (that depend on $\gamma$) such that, if $\delta_1 \ge \delta$, then
\[
\Ex^\eta \left[ X \cond \Circ_\delta(Q) \cap \Piv^q_Q(\cross(R,R)) \cap \{ s^{int}(\mathring{Q},R) \geq \delta_1 r \} \right] \geq 1 - C' (\delta/\delta_1)^{\eps'} \, .
\]
These two inequalities and the fact that $ \Circ_\delta(Q)$ is $\Prob^\eta$-independent of $\Piv^q_Q(\cross(R,R))$ and of $\Piv^q_Q(\cross(R,R)) \cap \{ s^{int}(\mathring{Q},R) \geq \delta_1 r \}$ imply that
\begin{multline*}
\Ex^\eta \left[ X \cond \Circ_\delta(Q) \cap \Piv^q_Q(\cross(R,R)) \right]\\
\geq  \Ex^\eta \left[ X \cond \Circ_\delta(Q) \cap \Piv^q_Q(\cross(R,R)) \cap \{ s^{int}(\mathring{Q},R) \geq \delta_1 r \} \right]\\
\times \Prob^\eta \left[ s^{int}(\mathring{Q},R) \geq \delta_1 r \cond \Piv_Q^q(\cross(R,R)) \right]\\
\geq (1 - C' (\delta/\delta_1)^{\eps'} ) (1 - C\delta_1^\eps/c) \, .
\end{multline*}
This implies Item~iii) if $\delta > 0$ is sufficiently small (for instance by choosing $\delta_1=\sqrt{\delta}$). The proof of Item~iv) is the same. This ends the proof of the lemma since, as noted in the proof of Proposition~\ref{p.quenched_QM}, $\Pro \left[ \gp^\gamma_{\delta/100}(r) \right] \geq 1-C''r^{-\gamma}$ for some $C''=C''(\delta,\gamma)<+\infty$.}
\end{proof}

\section{Proof that $\alpha^{an}_j(r,R) \asymp \widetilde{\alpha}_j(r,R)$}\label{s.asymp}

In this section, we prove~\eqref{e.main2} of Theorem~\ref{t.quenched_arm}, i.e.\ we show that there exists a constant $C = C(j) < +\infty$ such that, for every $1 \leq r \leq R < +\infty$,
\[
\alpha_j^{an}(r,R) \leq \widetilde{\alpha}_j(r,R) \leq C \, \alpha_j^{an}(r,R) \, .
\]
\begin{proof}[Proof of~\eqref{e.main2} \red{from} Theorem~\ref{t.quenched_arm}] Let us first note that, by the quasi-multiplicativity property and~\eqref{e.poly}, it is sufficient to prove the result for $r$ sufficiently large and $r \leq R/2$. Let $j \in \N^*$ and let $r_0=r_0(j)<+\infty$ to be fixed later. We actually prove the following \blue{more quantitative} result: There exist $h > 0$ and $C=C(j)<+\infty$ such that, if $r_0$ is sufficiently large and if $r_0 \leq r \leq R/2$, then
\begin{equation}\label{e.sufficient_an_et_tilde_first}
0 \leq \widetilde{\alpha}_j(r,R)^2 - \alpha^{an}_j(r,R)^2 \leq C r^{-h} \alpha^{an}_j(r,R)^2 \, .
\end{equation}
First note that it is sufficient to prove that there exist $h > 0$ and $C'=C'(j)<+\infty$ such that, if $r_0$ is sufficiently large and if $r_0 \leq r \leq R/2$, then
\begin{equation}\label{e.sufficient_an_et_tilde}
0 \leq \widetilde{\alpha}_j(r,R)^2 - \alpha^{an}_j(r,R)^2 \leq C' r^{-h} \widetilde{\alpha}_j(r,R)^2 \, .
\end{equation}
Indeed, this implies~\eqref{e.sufficient_an_et_tilde_first} with $C=2C'$ if $r_0$ satisfies $C'r_0^{-h} \leq 1/2$.\\ 

Let us prove~\eqref{e.sufficient_an_et_tilde}. If we apply Proposition~\ref{p.martingale} to $E = \arm_j(r,R)$ and $\rho=2$, we obtain that
\[
\Var \left( \Prob^\eta \left[ \arm_j(r,R) \right] \right) \leq \sum_{S \text{ square of the grid } 2\Z^2} \E \left[ \Prob^\eta \left[ \Piv_S(\arm_j(r,R)) \right]^2 \right] \, .
\]
Let us use Lemmas~\ref{l.piv_arm_1} to~\ref{l.piv_arm_5} to estimate the right-hand-side of this inequality. We will also need the following three estimates on arm events (see Propositions~\ref{p.alpha_4} and~\ref{p.universal} and~\eqref{e.poly}):
\begin{equation}\label{e.alpha_4}
\widetilde{\alpha}_4(\rho) \leq \grandO{1}\rho^{-(1+\epsilon)} \, ,
\end{equation}
\begin{equation}\label{e.3arm}
\widetilde{\alpha}_3^{++}(\rho,\rho')  \leq \widetilde{\alpha}_3^+(\rho,\rho') \leq \grandO{1} \frac{\rho}{\rho'} \,  ,
\end{equation}
\begin{equation}\label{e.poly_()c}
\widetilde{\alpha}_3^{(++)^c}(\rho,\rho') \leq \grandO{1} \left( \frac{\rho}{\rho'} \right)^{\epsilon/2} \, .
\end{equation}
\blue{(The exponent $\epsilon/2$ above is only to simplify the calculations.)} We can (and do) assume that $\epsilon < 1/2$, which will make the calculations easier. Below, we use several times the quasi-multiplicativity property Proposition~\ref{p.QM} and the polynomial decay property~\eqref{e.poly} without mentioning it. Note that a  difference compared to similar calculations for Bernoulli percolation on $\Z^2$ or on the triangular lattice is that we do not know that the contribution of the $3$-arm event in the half-plane from scale $\rho$ to scale $\rho'$ is $(\rho/\rho')^2$: we only have the upper bound~\eqref{e.3arm}.

By Lemma~\ref{l.piv_arm_1}, the contribution of the boxes $S$ in $A(2r,R/2)$ is at most (where $2^k$ has to be thought of as the order of the distance between the box and $0$)
\begin{eqnarray*}
\sum_{k=\log_2(r)}^{\log_2(R)} 2^{2k} \, \widetilde{\alpha}_j(r,R)^2 \, \widetilde{\alpha}_4(2^k)^2 \leq \grandO{1} \,\widetilde{\alpha}_j(r,R)^2 \, r^{-2\epsilon} \text{ (by } \eqref{e.alpha_4}) \, .
\end{eqnarray*}

By Lemma~\ref{l.piv_arm_3}, we can estimate the contribution of the boxes outside of $B_{R/2}$ by summing only on the boxes that intersect $A(R/2,R)$. By Lemma~\ref{l.piv_arm_2}, the contribution of such boxes is at most (where $2^k$ has to be thought of as the order of the distance between the box and $\partial B_R$)
\begin{align*}
& \sum_{k=0}^{\log_2(R)} 2^k R \, \widetilde{\alpha}_j(r,R)^2 \, \widetilde{\alpha}_4(2^k)^2 \, \widetilde{\alpha}_3^+(2^k,R)^2\\
& \leq \grandO{1} \widetilde{\alpha}_j(r,R)^2 \sum_{k=0}^{\log_2(R)} 2^k R \, 2^{-2k(1+\epsilon)} \left( \frac{2^k}{R} \right)^2 \text{ (by }\eqref{e.alpha_4} \text{ and } \eqref{e.3arm})\\
& \leq \grandO{1} \widetilde{\alpha}_j(r,R)^2 \sum_{k=0}^{\log_2(R)} \frac{2^{k(1-2\epsilon)}}{R} \nonumber\\
& \leq \grandO{1} \widetilde{\alpha}_j(r,R)^2 R^{-2\epsilon} \, .
\end{align*}

The contribution of the boxes in $B_{2r}$ is a little more difficult to estimate. By Lemma~\ref{l.piv_arm_5}, we can estimate the contribution of these boxes by summing only on the boxes that intersect $A(r,2r)$. To estimate the contribution of such boxes, we can use Lemma~\ref{l.piv_arm_4} and we obtain the following: (here, $2^k$ has to the thought of as the order of the distance between the box and $\partial B_r$ and $2^j$ has to be thought of as the distance between \red{the projection of the box on $\partial B_r$} and the nearest corner of $B_r$)
\begin{multline*}
\sum_{k=0}^{\log_2(r)} \sum_{j=k}^{\log_2(r)} 2^{k+j} \, \widetilde{\alpha}_j(r,R)^2 \, \widetilde{\alpha}_4(2^k)^2 \, \widetilde{\alpha}_3^+(2^k,2^j)^2 \, \widetilde{\alpha}_3^{(++)^c}(2^j,r)^2\\
+ \sum_{k=0}^{\log_2(r)} \sum_{j=0}^{k} 2^{k+j} \, \widetilde{\alpha}_j(r,R)^2 \, \widetilde{\alpha}_4(2^k)^2 \, \widetilde{\alpha}_3^{(++)^c}(2^k,r)^2.
\end{multline*}
The second sum is of the same order as the first sum restricted to the terms that satisfy $j=k$, so it is sufficient to estimate the first sum, which is less than or equal to
\begin{align*}
& \grandO{1} \widetilde{\alpha}_j(r,R)^2 \sum_{k=0}^{\log_2(r)} 2^k \, 2^{-2k(1+\epsilon)} \sum_{j=k}^{\log_2(r)} 2^j \, \left( \frac{2^k}{2^j} \right)^2 \, \left( \frac{2^j}{r} \right)^\epsilon \text{ (by } \eqref{e.alpha_4}\text{, } \eqref{e.3arm} \text{ and } \eqref{e.poly_()c}) \\
& \leq \grandO{1} \widetilde{\alpha}_j(r,R)^2 r^{-\epsilon}  \sum_{k=0}^{\log_2(r)} 2^k \, 2^{-2k(1+\epsilon)} 2^{k(1+\epsilon)}\\
& = \grandO{1} \widetilde{\alpha}_j(r,R)^2 r^{-\epsilon} \sum_{k=0}^{\log_2(r)} 2^{-k\epsilon}\\
& \leq \grandO{1} \widetilde{\alpha}_j(r,R)^2 r^{-\epsilon} \, .
\end{align*}
Altogether,
\[
\widetilde{\alpha}_j(r,R)^2-\alpha_j^{an}(r,R)^2 = \Var \left( \Prob^\eta \left[ \arm_j(r,R) \right] \right) \leq \grandO{1} \widetilde{\alpha}_j(r,R)^2 r^{-\epsilon}  \, ,
\]
and thus the estimate~\eqref{e.sufficient_an_et_tilde} is proved, which ends the proof.
\end{proof}

\begin{rem}
By exactly the same proof (i.e.\ by proving analogues of Lemmas~\ref{l.piv_arm_1} to~\ref{l.piv_arm_5} for arm events in \red{the half-plane etc.}), we obtain~\eqref{e.main2} of Theorem~\ref{t.quenched_arm} also for the quantities $\alpha_k^+(\cdot,\cdot)$, $\alpha_k^{++}(\cdot,\cdot)$ and $\alpha_k^{(++)^c}(\cdot,\cdot)$.
\end{rem}

\section{Strict inequality for the exponent of the annealed percolation function}\label{s.strict}

Let us prove Theorem~\ref{t.strict} by using the scaling relations from~\cite{scaling_voro} and the estimate~\eqref{e.main2} from Theorem~\ref{t.quenched_arm}. The estimate~\eqref{e.main2} will be used to prove the following:
\begin{prop}\label{p.alpha_4_geq}
There exists $\epsilon > 0$ such that, for every $1 \leq r \leq R < +\infty$,
\[
\alpha_4^{an}(r,R) \geq \epsilon \, \frac{1}{\alpha_1^{an}(r,R)} \left( \frac{r}{R} \right)^{2-\epsilon} \, .
\]
\end{prop}
Let us first explain why Proposition~\ref{p.alpha_4_geq} (with $r=1$) and Theorem~\ref{t.scaling_rel} imply Theorem~\ref{t.strict}.
\begin{proof}[Proof of Theorem~\ref{t.strict}]
By the two scaling relations of Theorem~\ref{t.scaling_rel}, we have
\[
\theta(p) \frac{1}{p-1/2} \asymp \alpha_1^{an}(L(p)) L(p)^2 \alpha_4^{an}(L(p)) \, .
\]
Since we know that $L(p)$ goes to $+\infty$ polynomially fast in $\frac{1}{p-1/2}$ as $p$ goes to $1/2$ (see Subsection~1.4 of~\cite{scaling_voro}), it is sufficient to prove that $\alpha_1^{an}(L(p)) L(p)^2 \alpha_4^{an}(L(p)) \geq \Omega(1) L(p)^\epsilon$ for some $\epsilon > 0$, which is given by Proposition~\ref{p.alpha_4_geq}.
\end{proof}

\begin{proof}[Proof of Proposition~\ref{p.alpha_4_geq}]
We follow Appendix~A of~\cite{garban2010fourier}, where the analogous result is proved for Bernoulli percolation on $\Z^2$ by Beffara. Let $M \geq 100$ and let $\rho \geq M$. Also, let $\gp(\rho,M)$ be defined as follows:
\[
\gp(\rho,M) = \bigcap_{k=0}^{\lfloor \log_5 \left( M \right) \rfloor-1} \dense_{1/100} \left( A(5^k\rho,10 \cdot 5^k \rho) \right) \cap \qbc_{1/100}^3 \left( A(5^k\rho,10 \cdot 5^k \rho) \right)
\]
(where the events ``$\dense$'' and ``$\qbc$'' are the events defined in Subsection~\ref{ss.good}; $\gp$ means ``Good Point process''). By Lemma~\ref{l.dense} and Remark~\ref{r.a_lot_of_quads}, we have
\[
\Pro \left[ \gp(\rho,M) \right] \geq 1-\grandO{1} \rho^{-3} \, .
\]
If $\eta \in \gp(\rho,M)$ and if we follow the beginning of Appendix~A of~\cite{garban2010fourier} (where the authors study the winding number of arms), we obtain that (if $M$ is sufficiently large):
\[
\Prob^\eta \left[ \arm_5(\rho,M\rho) \right] \leq M^{-\epsilon} \, \Prob^\eta \left[ \arm_1(\rho,M \rho) \right] \, \Ex^\eta \left[ Y^3 \un_{Y \geq 4} \right] \, ,
\]
where $Y$ is the number of interfaces from $\partial B_{\rho}$ to $\partial B_{M \rho}$ and where $\epsilon \in (0,1)$ depends only on the box-crossing constant $c=c(1/100,3)$ from Remark~\ref{r.a_lot_of_quads}. Indeed, the fact that $\eta \in \gp(\rho,M)$ implies that we can apply the independence arguments and the box-crossing arguments from Appendix~A of~\cite{garban2010fourier}.\\
Still as in Appendix~A of~\cite{garban2010fourier}, we have $\Ex^\eta \left[ Y^3 \un_{Y \geq 4} \right] \leq C \, \Prob^\eta \left[ \arm_4(\rho,M\rho) \right]$ for some $C<+\infty$ that depends only on the constant $c=c(1/100,3)$ from Remark~\ref{r.a_lot_of_quads}. Indeed, what is used in~\cite{garban2010fourier} to prove this estimate is Reimer's inequality (that holds for the quenched probability measure $\Prob^\eta$) and the fact that i) $\Prob^\eta \left[ \arm_1(\rho,M \rho) \right] \leq M^{-a}$ and ii) $\Prob^\eta \left[ \arm_4(\rho,M\rho) \right] \geq M^{-b}$ for some $a,b \in (0,+\infty)$. The properties i) and ii) follow from classical box-crossing arguments that we can use since $\eta \in \gp(\rho,M)$. Altogether,
\[
\alpha^{an}_5(\rho, M \rho) = \E \left[ \Prob^\eta \left[ \arm_5(\rho,M\rho) \right] \right] \leq C M^{-\epsilon} \, \E \left[ \Prob^\eta \left[ \arm_1(\rho,M \rho) \right] \Prob^\eta \left[ \arm_4(\rho,M\rho) \right] \right] + \grandO{1} \rho^{-3} \, .
\]
If we apply the Cauchy-Schwarz inequality and if we use Proposition~\ref{p.universal} to estimate the probability of the $5$-arm event, we obtain that
\begin{eqnarray*}
M^{-2} \asymp \alpha^{an}_5(\rho, M \rho) & \leq & C M^{-\epsilon}\sqrt{\E \left[\Prob^\eta \left[ \arm_1(\rho,M \rho) \right]^2 \right] \E \left[ \Prob^\eta \left[ \arm_4(\rho,M\rho) \right]^2 \right]} + \grandO{1} \rho^{-3}\\
& = & C M^{-\epsilon} \widetilde{\alpha}_1(\rho,M \rho) \widetilde{\alpha}_4(\rho,M\rho) + \grandO{1} \rho^{-3} \, .
\end{eqnarray*}
By~\eqref{e.main2} of Theorem~\ref{t.quenched_arm}, the quantities $\alpha^{an}_j(\cdot,\cdot)$ are of same order as the quantities $\widetilde{\alpha}_j(\cdot,\cdot)$; hence, the above implies that there exists $\epsilon'>0$ such that, if $M$ is sufficiently large, then for every $\rho \geq M$,
\[
M^{-2} \leq M^{-\epsilon'} \alpha_1^{an}(\rho,M \rho) \alpha_4^{an}(\rho,M\rho) \, .
\]
Now, the proof is a direct consequence of the quasi-multiplicativity property.
\end{proof}

\begin{rem}
Note that, if we follow the proof of Proposition~\ref{p.alpha_4_geq}, we obtain the following for every $j \in \N^*$: $\alpha_{2j+1}^{an}(r,R) \leq \grandO{1} \left( \frac{r}{R} \right)^{\Omega(1)} \alpha_1^{an}(r,R) \, \alpha_{2j}^{an}(r,R)$, where the constants in $\grandO{1}$ and $\Omega(1)$ only depend on $j$.
\end{rem}

\section{Other estimates on arm events}\label{s.other}

The last goal of this paper is to obtain the quantitative estimates~\eqref{e.main} from Theorem~\ref{t.quenched_arm} and~\eqref{e.maincross} from Theorem~\ref{t.quenched_cross}. In order to prove these results, we need two other estimates on arm events that we prove in this section. We have (see Section~\ref{s.estimates} for the notation $\alpha^{an,(++)^c}_3(r,R)$):
\begin{lem}\label{l.()c}
Let $1 \leq r \leq R$. Then,
\[
\alpha^{an,(++)^c}_3(r,R) \leq \grandO{1} \frac{r}{R} \, .
\] 
\end{lem}
\begin{proof}
Let $N=\lfloor R/(4r) \rfloor$ and let $Q_1=B_r, Q_2, \cdots, Q_N$ be the $2r \times 2r$ squares defined in Figure~\ref{f.()c} (note that these squares are included in $B_{R/2}$). For every $j \in \{1,\cdots,N\}$, we define the following event: $B(r,R;j)$ is the event that there exist two paths $\gamma_1$ and $\gamma_2$ such that: i) $\gamma_1$ is a black path included in $B_R$ from the left side of $B_R$ to its right side, ii) $\gamma_2$ is a white path included in $B_R$ from $\partial Q_j$ to the top side of $B_R$, iii) $\gamma_1$ and $\gamma_2$ do not intersect the quarter plane $\{ x_j+(a,b) \, : \, a,b \leq 0 \}$ where $x_j$ is the center of $Q_j$.

\begin{figure}[!h]
\begin{center}
\includegraphics[scale=0.485]{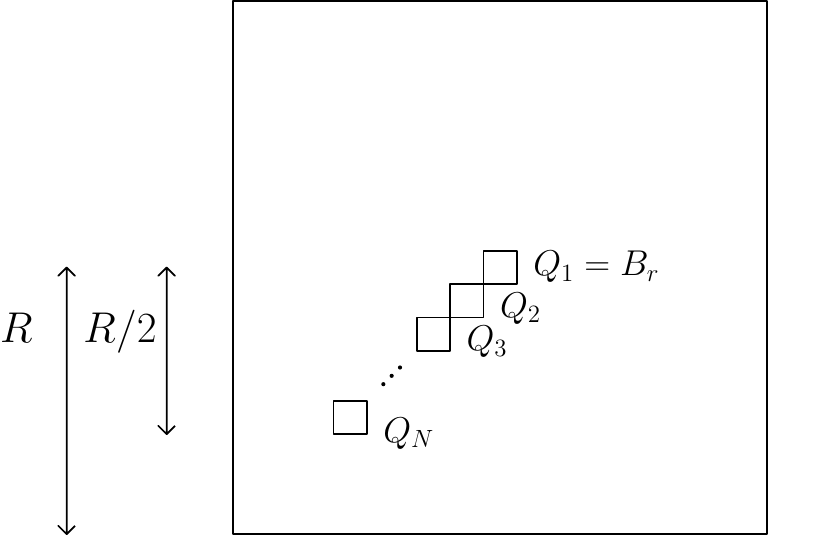}
\includegraphics[scale=0.485]{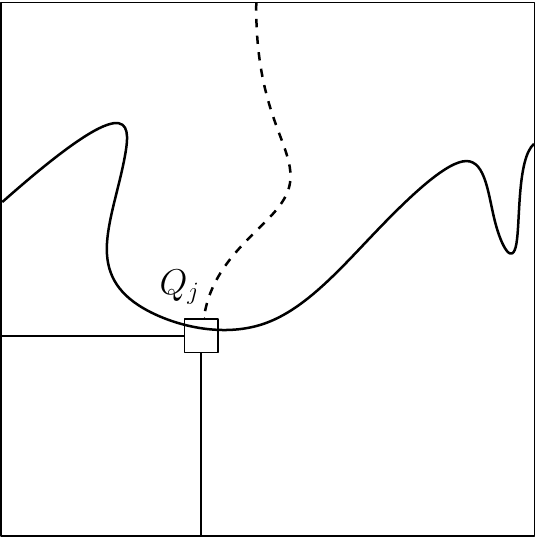}
\end{center}
\caption{The $2r \times 2r$ squares $Q_j$ and the events $B(r,R;j)$.}\label{f.()c}
\end{figure}
Note that the events $B(r,R;1), \cdots, B(r,R;N)$ are pairwise disjoint, hence
\[
\inf_{j=1}^N \Pro \left[ B(r,R;j) \right] \leq \frac{1}{N} \, .
\]
As a result, it is sufficient for our purpose to prove that $\Pro \left[ B(r,R;j) \right] \geq \Omega(1) \alpha_3^{(++)^c}(r,R)$ where the constants in $\Omega(1)$ are absolute constants. For Bernoulli percolation on $\Z^2$ or on the triangular lattice, this comes from separation of arms results. For Voronoi percolation, we have proved separation of arm results in~\cite{scaling_voro} and we have deduced for instance that $\alpha^{an}_4(r,R)$ is at most some constant times the probability that there exist two black paths from $\partial B_r$ to the left and right sides of $B_R$ and two white paths from $\partial B_r$ to the top and bottom sides of $B_R$ (see Lemma~4.3 therein). Since the proof that $\Pro \left[ B(r,R;j) \right] \geq \Omega(1) \alpha_3^{(++)^c}(r,R)$ is the same, we refer to~\cite{scaling_voro} and leave the details to the reader (the only difference is that we need to use Proposition~2.5 of~\cite{scaling_voro} for arm events in the plane without the quarter plane instead of for arm events in the plane, but the proof is the same).
\end{proof}

\blue{To prove the following result, we rely a lot on the quenched properties from Section~\ref{s.estimates}. The main difficulty (compared to Proposition \ref{p.alpha_4}) is that this is a multiscale estimate. See Appendix B of \cite{schramm2011scaling} for the proof of this result for bond percolation on $\Z^2$.}
\begin{prop}\label{p.alpha_4_from_SS}
For every $\epsilon > 0$ there exists $C=C(\epsilon) < +\infty$ such that for every $1 \leq r \leq R$ we have
\[
\alpha_4^{an}(r,R) \leq C \, \left( \frac{r}{R} \right)^{1-\epsilon} \sqrt{\alpha^{an}_2(r,R)} \, .
\]
In particular, there exists $\delta > 0$ such that for every $1 \leq r \leq R$ we have
\[
\alpha_4^{an}(r,R) \leq \frac{1}{\delta} \left( \frac{r}{R} \right)^{1+\delta} \, .
\]
\end{prop}
\begin{proof}
We follow the proof of the analogous result for bond percolation on $\Z^2$ by Garban from Appendix~B of~\cite{schramm2011scaling}. To this purpose, we use both the annealed quasi-multiplicativity property Proposition~\ref{p.QM} and the quenched properties from Subsection~\ref{ss.quenched_prop}. We let $M \in [100,+\infty)$ to be fixed later and we consider some $\rho \in [10 M,+\infty)$. Note that, by the annealed quasi-multiplicativity property, it is sufficient to prove that, if $M$ is sufficiently large, then
\begin{equation}\label{e.sufficient_alpha_4}
\alpha^{an}_4(\rho,\rho M) \leq \grandO{1} M^{-1} \sqrt{\alpha^{an}_2(\rho,\rho M)} \, .
\end{equation}
Let us prove this estimate. We need the following notations: \red{we let $Q_1,\cdots,Q_N$ be the $\rho \times \rho$ squares of the grid $\rho \Z^2$ that are included in the square $B_{\rho M}$ and are at distance at least $\rho M / 3$ from the sides of this square. Note that $N \asymp M^2$.} We also write $X$ for the $\pm 1$ indicator function of $\cross(\rho M, \rho M)$. If $\alpha \in (0,1)$, we let $\alpha Q_j$ denote the square concentric to $Q_j$ with side length $\alpha \rho$. Also for $\delta \in (0,1)$, we let $C_\delta(j)$ denote the random variable that equals
\bi 
\item $1$ if there is a black circuit in the annulus $A_j(\delta):=(1-\delta)Q_j \setminus (1-2\delta)Q_j$ and no white circuit in $A_j(\delta)$,
\item $-1$ if there is a white circuit in the annulus $A_j(\delta):=(1-\delta)Q_j \setminus (1-2\delta)Q_j$ and no black circuit in $A_j(\delta)$,
\item $0$ otherwise.
\ei
Note that $\Ex^\eta \left[ C_\delta(j) \right] = 0$ for every $j$ and $\eta$. Let $\gamma$ be some sufficiently large constant to be fixed later. Write $x_j$ for the center of $Q_j$ and, for every $x \in \R^2$ and every $k \in \N^*$, let $\arm_k(x;\cdot,\cdot)$ be the $k$-arm event $\arm_k(\cdot,\cdot)$ translated by $x$. By Lemma~\ref{l.lemme_for_alpha4} (and by \red{the union bound}), we can choose $\delta$ sufficiently small so that, with probability at least $1-C M^2 \rho^{-\gamma}$, for every $j$ we have
\begin{eqnarray}
\Prob^\eta \left[ \Piv^q_{Q_j}(\cross(\rho M,\rho M)) \right] & \geq & c \Prob^\eta \left[ \arm_4(x_j;\rho,\rho M) \right] \, ,\label{e.poiu1}\\
\Prob^\eta \left[ C_\delta(j)=-1 \right] = \Prob^\eta \left[ C_\delta(j)=1 \right] & \geq & c' \, ,\label{e.poiu2}\\
\Ex^\eta \left[ X \cond \{ C_\delta(j)=1 \} \cap \Piv^q_{Q_j}(\cross(\rho M,\rho M)) \right]&  > & 1/4 \, ,\label{e.poiu3}\\
\Ex^\eta \left[ X \cond \{ C_\delta(j)=-1 \} \cap \Piv^q_{Q_j}(\cross(\rho M,\rho M)) \right]&  < & -1/4 \, ,\label{e.poiu4}
\end{eqnarray}
for some constants $C=C(\delta,\gamma)$, $c=c(\gamma)$ and $c'=c'(\delta,\gamma)$. We fix such a $\delta$. Below, the constants in the $\grandO{1}$'s and $\Omega(1)$'s may depend on $\delta$ and $\gamma$. Next, we define the following event:
\[
\dense_\delta(\rho,M) = \bigcap_{Q \text{ square of the grid } \rho \Z^2 \text{ included in } B_{\rho M}} \dense_{\delta/100}(Q) \, .
\]
By Lemma~\ref{l.dense}, $\Pro \left[ \dense_\delta(\rho,M) \right] \geq 1-\grandO{1} M^2 \exp(-\Omega(1)\rho^2)$. Now, assume that $\eta \in \dense_\delta(\rho,M)$ and that $\eta$ is such that~\eqref{e.poiu1} to~\eqref{e.poiu4} hold, and let us explain how we can follow Appendix~B of~\cite{schramm2011scaling} in order to obtain that
\begin{equation}\label{e.M_and_alpha2}
\sum_j \Prob^\eta \left[ \arm_4(x_j;\rho,\rho M) \right] \leq \grandO{1} \sqrt{ \sum_j \Prob^\eta \left[ \arm_2(x_j;3\rho, \rho M /3) \right]} \, .
\end{equation}
\begin{proof}[Proof of~\eqref{e.M_and_alpha2}]
In this proof, we fix some $\eta \in \dense_\delta(\rho,M)$ that satisfies~\eqref{e.poiu1} to~\eqref{e.poiu4} and we work under the probability measure $\Prob^\eta$.

As in Appendix~B of~\cite{schramm2011scaling}, we look at the interface that goes from the cell that contains the top-right corner of $B_{\rho M}$ to the cell that contains the bottom-right corner of this square, with black boundary condition on the right side and white boundary condition on the other sides. Moreover, we let $Y_j$ be the indicator of the event that the distance between the interface and $Q_j$ is at most $\rho$. \blue{Since $\eta \in \dense_\delta(\rho,M)$, $X$ is independent of $C_\delta(j)$ on $\{ Y_j = 0 \}$ and $C_\delta(j)$ is independent of $Y_j$, hence
\begin{align*}
\Ex^\eta \left[ X C_\delta(j) Y_j \right]& =  \Ex^\eta \left[ X C_\delta(j) \right] - \Ex^\eta \left[ X C_\delta(j) \un_{Y_j = 0} \right]\\
&= \Ex^\eta \left[ X C_\delta(j) \right] - \red{ \frac{\Ex^\eta \left[ X \un_{Y_j=0} \right] \Ex^\eta \left[  C_\delta(j) \un_{Y_j = 0} \right]}{\Prob^\eta \left[ Y_j = 0 \right]}} \\
&= \Ex^\eta \left[ X C_\delta(j) \right] - \Ex^\eta \left[ X \un_{Y_j=0} \right] \Ex^\eta \left[  C_\delta(j) \right]\\
&= \Ex^\eta \left[ X C_\delta(j) \right] \, ,
\end{align*}
where the last equality holds because $\Ex^\eta \left[  C_\delta(j) \right]=0$.}

\blue{Next, we observe that $X$ is independent of $C_\delta(j)$ on $\{ \neg \Piv_{Q_j}^q(\cross(\rho M,\rho M)) \}$ and $C_\delta(j)$ is independent of $\Piv_{Q_j}^q(\cross(\rho M,\rho M))$. This observation and the fact that $\Ex^\eta \left[  C_\delta(j) \right]=0$ imply that
\begin{multline*}
\Ex^\eta \left[ X C_\delta(j) Y_j \right]=\Ex^\eta \left[ X C_\delta(j) \right]\\
=\Prob^\eta \left[ \Piv^q_{Q_j}(\cross(\rho M,\rho M)) \right] \Ex^\eta \left[ X C_\delta(j) \cond \Piv^q_{Q_j}(\cross(\rho M,\rho M)) \right].
\end{multline*}
By using once again that $C_\delta(j)$ is independent of $\Piv_{Q_j}^q(\cross(\rho M,\rho M))$, we obtain that the above equals
\begin{align*}
&\Prob^\eta \left[ \Piv^q_{Q_j}(\cross(\rho M,\rho M)) \right] \Big( \Ex^\eta \left[ X \cond \{ C_\delta(j) = 1 \} \cap \Piv^q_{Q_j}(\cross(\rho M,\rho M)) \right] \Pro \left[ C_\delta(j) = 1 \right]\\
& \hspace{1em} -\Ex^\eta \left[ X \cond \{ C_\delta(j) = -1 \} \cap \Piv^q_{Q_j}(\cross(\rho M,\rho M)) \right] \Pro \left[ C_\delta(j) = -1 \right] \Big) \, .
\end{align*}}
By~\eqref{e.poiu1} to~\eqref{e.poiu4}, this implies that
\[
\Ex^\eta \left[ X C_\delta(j) Y_j \right]  \geq \Omega(1) \Prob^\eta \left[ \arm_4(x_j;\rho,\rho M) \right] \, .
\]
As in~\cite{schramm2011scaling}, one also has
\[
\Ex^\eta \left[ C_\delta(i) Y_i C_\delta(j) Y_j \right] = 0 \text{ if } i \neq j \, .
\]
Indeed, we can let $k \in \{ i,j \}$ be such that the interface reaches the $\rho$-neighbourhood of $Q_k$ before the $\rho$-neighbourhood of $Q_l$ where $\{l\} = \{ i,j \} \setminus \{ k \}$, we can write $\mathcal{G}$ for the $\sigma$-algebra \red{(on $\{-1,1\}^\eta$; recall that we work under the probability measure $\Prob^\eta$)} generated by the colours of the Voronoi cells visited by the interface until it reaches the $\rho$-neighbourhood of $Q_l$ and by the colours of the Voronoi cells in $Q_k$, and we can note that $Y_i,Y_j,C_\delta(k)$ are $\mathcal{G}$-measurable and that $C_\delta(l)$ is independent of $\mathcal{G}$.
%Note that this is the only place where we use the fact that $(Q_j)_j$ is not the family of all the squares of the grid $\rho \Z^2$ ``in the bulk'' of $B_{\rho M}$ but that we have rather considered only half of them so we can give meaning.
As in~\cite{schramm2011scaling}, we can then apply the Cauchy-Schwarz inequality to obtain that
\[
\sum_j \Ex^\eta \left[ X C_\delta(j) Y_j \right] \leq \grandO{1} \sqrt{ \sum_j \Ex^\eta \left[ Y_j \right]} \, .
\]
This ends the proof since $\Ex^\eta \left[ Y_j \right] \leq \Prob^\eta \left[ \arm_2(x_j;3\rho,\rho M/3) \right]$.
\end{proof}
If we take the expectation of the left and right sides of~\eqref{e.M_and_alpha2}, we obtain that
\begin{multline*}
\sum_j \E \left[ \Prob^\eta \left[ \arm_4(x_j;\rho,\rho M) \right] \right]\\
 \leq \grandO{1} \E \left[ \sqrt{\sum_j \Prob^\eta \left[ \arm_2(x_j;3\rho,\rho M/3) \right] } \right] + \grandO{1} M^2 \rho^{-\gamma} + \grandO{1} M^2 \exp(-\Omega(1)\rho^2) \, .
\end{multline*}
By Jensen's inequality and since $\rho \geq M$, we have
\begin{multline*}
\sum_j \E \left[ \Prob^\eta \left[ \arm_4(x_j;\rho,\rho M) \right] \right]\\
 \leq \grandO{1} \sqrt{\sum_j \E \left[ \Prob^\eta \left[ \arm_2(x_j;3\rho,\rho M/3) \right] \right]} + \grandO{1} M^{2-\gamma} + \grandO{1} M^2 \exp(-\Omega(1)M^2)
\end{multline*}
i.e.\ (by translation invariance of the annealed probability measure)
\[
M^2 \alpha_4^{an}(\rho,\rho M) \leq \grandO{1} M \sqrt{ \alpha^{an}_2(3 \rho,\rho M/3)} + \grandO{1} M^{2-\gamma} + \grandO{1} M^2 \exp(-\Omega(1)M^2) \, .
\]
Since the probabilities of arm events decay polynomially fast (see~\eqref{e.poly}), we can choose $\gamma$ sufficiently large so that for every sufficiently large $M$ we have $\grandO{1} M^{-\gamma} + \grandO{1} \exp(-\Omega(1)M^2) \leq \grandO{1} M^{-1} \sqrt{\alpha^{an}_2(3\rho,\rho M/3)}$. For these choices of $M$ and $\gamma$ we obtain
\[
\alpha^{an}_4(\rho,\rho M) \leq \grandO{1} M^{-1} \sqrt{\alpha^{an}_2(3\rho,\rho M/3)} \leq \grandO{1} M^{-1} \sqrt{\alpha^{an}_2(\rho,\rho M)} \, ,
\]
where the second inequality is a direct consequence of~\eqref{e.poly} and of the quasi-multiplicativity property. This implies~\eqref{e.sufficient_alpha_4} and ends the proof of the proposition.
\end{proof}

\section{Quantitative quenched estimates}\label{s.quant}

Let us now prove~\eqref{e.main} of Theorem~\ref{t.quenched_arm} by using~\eqref{e.main2} and the estimates from Section~\ref{s.other}.
\begin{proof}[Proof of~\eqref{e.main} \red{from} Theorem~\ref{t.quenched_arm}]
The proof is very close to the proof of~\eqref{e.main2} from Theorem~\ref{t.quenched_arm}. The difference is that now we can use that the quantities $\widetilde{\alpha}_k(\cdot,\cdot)$ are of the same order as the quantities $\alpha^{an}_k(\cdot,\cdot)$. As a result, we can use Lemmas~\ref{l.piv_arm_1} to~\ref{l.piv_arm_5} with $\alpha^{an}_k(\cdot,\cdot)$ instead of $\widetilde{\alpha}_k(\cdot,\cdot)$. The estimates on arm events that we are going to use are the following (see~Propositions~\ref{p.alpha_4_from_SS},~\ref{p.alpha_4_2} and~\ref{p.universal} and Lemma~\ref{l.()c}):
\begin{equation}\label{e.2alpha_4}
\alpha_4^{an}(\rho,\rho') \leq \grandO{1} \left( \frac{\rho}{\rho'} \right)^{1+\epsilon} \, ,
\end{equation}
\begin{equation}\label{e.23arm}
\alpha^{an,++}_3(\rho,\rho') \leq \alpha^{an,+}_3(\rho,\rho') \asymp \left( \frac{\rho}{\rho'} \right)^2 \leq \grandO{1} \alpha^{an}_4(\rho,\rho') \left( \frac{\rho}{\rho'} \right)^{\epsilon} \, ,
\end{equation}
\begin{equation}\label{e.2poly_()c}
\alpha_3^{an,(++)^c}(\rho,\rho') \leq \grandO{1} \frac{\rho}{\rho'} \, ,
\end{equation}
for some $\epsilon > 0$.\\
If we apply Proposition~\ref{p.martingale} to $E = \arm_j(r,R)$ and $\rho=2$, we obtain that
\[
\Var \left( \Prob^\eta \left[ \arm_j(r,R) \right] \right) \leq \sum_{S \text{ square of the grid } 2\Z^2} \E \left[ \Prob^\eta \left[ \Piv_S(\arm_j(r,R)) \right]^2 \right] \, .
\]
Let us now use Lemmas~\ref{l.piv_arm_1} to~\ref{l.piv_arm_5} with $\alpha^{an}_k(\cdot,\cdot)$ instead of $\widetilde{\alpha}_k(\cdot,\cdot)$. As in the proof of~\eqref{e.main2}, we use the quasi-multiplicativity property and the polynomial decay property without mentioning it. By the same considerations as in the proof of \eqref{e.main2}, we obtain that the contribution of the boxes $S$ in $A(2r,R/2)$ is at most
\begin{eqnarray*}
\sum_{k=\log_2(r)}^{\log_2(R)} 2^{2k} \, \alpha^{an}_j(r,R)^2 \, \alpha^{an}_4(2^k)^2 & \leq & \grandO{1} \,\alpha^{an}_j(r,R)^2 \, \alpha_4(r)^2 \sum_{k=\log_2(r)}^{\log_2(R)} 2^{2k} \alpha_4(r,2^k)^2\\
& \leq & \grandO{1} \,\alpha^{an}_j(r,R)^2 \, \alpha_4(r)^2 r^2 \text{ by~\eqref{e.2alpha_4}} \, .
\end{eqnarray*}
(Note that, in order to obtain the above estimate, Proposition~\ref{p.alpha_4} is not enough and we need the multiscale estimate Proposition~\ref{p.alpha_4_from_SS}.) The contribution of the boxes outside of $B_{R/2}$ is at most
\begin{multline*}
\sum_{k=0}^{\log_2(R)} 2^k R \, \alpha^{an}_j(r,R)^2 \, \alpha^{an}_4(2^k)^2 \, \alpha^{an,+}_3(2^k,R)^2\\
\leq \grandO{1} \alpha^{an}_j(r,R)^2 \, \alpha^{an}_4(R)^2 \sum_{k=0}^{\log_2(R)} 2^k R \frac{1}{\alpha_4^{an}(2^k,R)^2} \alpha^{an,+}_3(2^k,R)^2 \\
\leq \grandO{1} \alpha^{an}_j(r,R)^2 \, \alpha^{an}_4(R)^2 \sum_{k=0}^{\log_2(R)} 2^k R  \left( \frac{2^k}{R} \right)^{2\epsilon-4} \left( \frac{2^k}{R} \right)^4\\
\leq \grandO{1} \alpha^{an}_j(r,R)^2 \, R^2 \, \alpha^{an}_4(R)^2 \, .
\end{multline*}
The contribution of the boxes in $B_{2r}$ is at most
\begin{align*}
& \sum_{k=0}^{\log_2(r)} \sum_{j=k}^{\log_2(r)} 2^{k+j} \, \alpha^{an}_j(r,R)^2 \, \alpha^{an}_4(2^k)^2 \, \alpha^{an,+}_3(2^k,2^j)^2 \,\alpha^{an,(++)^c}_3(2^j,r)^2\\
& \leq \grandO{1} \alpha^{an}_j(r,R)^2 \sum_{k=0}^{\log_2(r)} 2^k \, \sum_{j=k}^{\log_2(r)} 2^j \alpha^{an}_4(2^j)^2 \frac{\alpha^{an,+}_3(2^k,2^j)^2}{\alpha^{an}_4(2^k,2^j)^2} \,\alpha^{an,(++)^c}_3(2^j,r)^2\\
& \leq \grandO{1} \alpha^{an}_j(r,R)^2 \sum_{k=0}^{\log_2(r)} 2^k \, \sum_{j=k}^{\log_2(r)} 2^j \alpha^{an}_4(2^j)^2 \left( \frac{2^k}{2^j} \right)^{2\epsilon} \, \left( \frac{2^j}{r} \right)^2\\
& = \grandO{1} r^{-2} \alpha^{an}_j(r,R)^2 \sum_{k=0}^{\log_2(r)} 2^{k(1+2\epsilon)} \, \sum_{j=k}^{\log_2(r)} 2^{j(3-2\epsilon)} \alpha^{an}_4(2^j)^2\\
& = \grandO{1} r^{-2} \alpha^{an}_j(r,R)^2 \sum_{j=0}^{\log_2(r)} 2^{j(3-2\epsilon)} \alpha^{an}_4(2^j)^2 \sum_{k=0}^{j} 2^{k(1+2\epsilon)} \\
& \leq \grandO{1} r^{-2} \alpha^{an}_j(r,R)^2 \sum_{j=0}^{\log_2(r)} 2^{4j} \alpha^{an}_4(2^j)^2 \\
& \leq \grandO{1} \alpha^{an}_j(r,R)^2 \, r^2 \, \alpha^{an}_4(r)^2 \, .
\end{align*}
Finally,
\begin{multline*}
\Var \left( \Prob^\eta \left[ \arm_j(r,R) \right] \right) \leq \grandO{1}\alpha^{an}_j(r,R)^2 \left( r^2 \, \alpha^{an}_4(r)^2 +  R^2 \, \alpha^{an}_4(R)^2 \right)\\ \leq \grandO{1}\alpha^{an}_j(r,R)^2 \, r^2 \, \alpha^{an}_4(r)^2 \, ,
\end{multline*}
which ends the proof.
\end{proof}

We end the paper by proving the quantitative quenched estimate Theorem~\ref{t.quenched_cross}.

\begin{proof}[Proof of Theorem~\ref{t.quenched_cross}]
If we apply Proposition~\ref{p.martingale} to $E=\cross(\lambda R,R)$ and $\rho=2$ we obtain that
\[
\Var \left( \Prob^\eta \left[ \cross(\lambda R,R) \right] \right) \leq \sum_{S \text{ square of the grid } 2\Z^2} \E \left[ \Prob^\eta \left[ \Piv_S(\cross(\lambda R,R)) \right]^2 \right] \, .
\]
By using analogues of Lemmas~\ref{l.piv_arm_1},~\ref{l.piv_arm_2},~\ref{l.piv_arm_3},~\ref{l.piv_arm_4} and~\ref{l.piv_arm_5} for crossing events and by using the fact that we know that the quantities $\widetilde{\alpha}_k(\cdot,\cdot)$ are of the same order as the quantities $\alpha^{an}_k(\cdot,\cdot)$ (i.e. by following the proof of~\eqref{e.main}), we obtain that this sum is less than or equal to
\[
\grandO{1} \, R^2 \, \alpha_4^{an}(R)^2 \, ,
\]
which ends the proof.
\end{proof}

\bibliographystyle{alpha}
\bibliography{ref_perco}

\ni
{\bf Hugo Vanneuville} \\
Univ. Lyon 1\\
UMR5208, Institut Camille Jordan, 69100 Villeurbanne, France\\
vanneuville@math.univ-lyon1.fr\\
\url{http://math.univ-lyon1.fr/~vanneuville/}\\
Supported by the ERC grant Liko No 676999\\

\end{document}